\pdfoutput=1
\documentclass[11pt]{amsart}

\usepackage{a4wide}
\usepackage[utf8]{inputenc}
\usepackage[T1]{fontenc}
\usepackage[english]{babel}

\usepackage[square,numbers]{natbib}

\usepackage{amssymb}
\usepackage{amsthm}
\usepackage{mathtools}
\usepackage{xcolor}

\usepackage{hyperref}
\hypersetup{            
	colorlinks=true,
	linkcolor=blue,
	filecolor=green,      
	urlcolor=cyan,
}
\usepackage{cleveref}
\usepackage{datetime}
\usepackage{enumitem}

\newcommand{\N}{\mathbb{N}}
\newcommand{\R}{\mathbb{R}}

\newcommand{\Chi}{\mathcal{X}}

\newcommand{\on}{\quad \text{on} \enspace}

\newcommand{\Set}[2]{\{\, #1 \ \mid #2 \, \}}
\newcommand{\bigSet}[2]{\big\{\, #1 \ \big| #2 \, \big\}}

\newcommand{\de}{\mathrm{d}}
\newcommand{\integralde}{\thinspace\mathrm{d}}
\newcommand{\integral}[1]{\int_{#1}}

\newcommand{\norm}[1]{\|#1\|}

\newcommand{\eps}{\varepsilon}
\newcommand{\union}[1][]{\bigcup_{#1}}
\newcommand{\intersection}[1][]{\bigcap_{#1}}

\theoremstyle{plain}

\newtheorem{thm}{Theorem}[section]
\newtheorem{cor}[thm]{Corollary}
\newtheorem{lemma}[thm]{Lemma}
\newtheorem{prop}[thm]{Proposition}

\theoremstyle{definition}
\newtheorem{defn}[thm]{Definition}

\theoremstyle{remark}
\newtheorem{remark}[thm]{Remark}

\newcommand{\sequence}[2]{\{#1\}_{#2}\ }

\newcommand{\unmezzo}{\frac{1}{2}}

\newcommand{\spazio}{\enspace}

\newcommand{\Wspace}[2]{W^{#1,#2}}

\newcommand{\Lspace}[1]{L^{#1}}
\newcommand{\Llocspace}[1]{L_{loc}^{#1}}

\newcommand{\BoundedVariation}{BV}

\DeclareMathOperator{\leb}{\mathcal{L}}
\newcommand{\lebn}{\leb^n}
\newcommand{\lebone}{\leb^1}

\newcommand{\closure}[2][]{\overline{#2}^{#1}}

\newcommand{\characteristic}[1]{\Chi_{#1}}
\newcommand{\uk}{u_k}

\newtheorem{innerclaim}{Claim}
\newenvironment{claim}[1]
{\renewcommand\theinnerclaim{#1}\innerclaim}
{\endinnerclaim}

\newcommand{\measurablesets}{\mathcal{M}_n}
\newcommand{\fractional}[1]{\mathcal{P}_{#1}}
\newcommand{\Ps}{\fractional{s}}
\newcommand{\GeometricProblem}[2]{\left({GP}^s_{#1,#2}\right)}
\newcommand{\GeometricProblemSolutions}[1]{\boldsymbol{GSol}^s\left(#1\right)}
\newcommand{\FunctionalProblem}[2]{\left({P}^s_{#1,#2}\right)}
\newcommand{\FunctionalProblemSolutions}[1]{\boldsymbol{Sol}^s\left(#1\right)}
\newcommand{\FunctionalProblemJumps}[1]{\boldsymbol{J}^s\left(#1\right)}
\newcommand{\maxi}[1]{{#1}^{+}}
\newcommand{\mini}[1]{{#1}^{-}}
\newcommand{\Emaxi}{\maxi{E}}
\newcommand{\Emini}{\mini{E}}

\newcommand{\Renne}{\R^n}
\newcommand{\Rtwon}{\R^{2n}}
\newcommand{\comp}[1]{{#1}^c} 
\newcommand{\Eone}{E_{1}}
\newcommand{\Etwo}{E_{2}}
\newcommand{\Eonemini}{\mini{\Eone}}
\newcommand{\Etwomini}{\mini{\Etwo}}
\newcommand{\Eonemaxi}{\maxi{\Eone}}
\newcommand{\Etwomaxi}{\maxi{\Etwo}}
\newcommand{\Ei}{E_{i}}
\newcommand{\Ej}{E_{j}}

\newcommand{\Etilde}{\tilde{E}}

\newcommand{\Fk}{F_{k}}

\newcommand{\onetwo}{\{1,2\}}
\newcommand{\forallionetwo}{\forall i\in\{1,2\}}

\newcommand{\BallRadiusCenter}[2]{B_{#1}({#2})}
\newcommand{\BallVolumeCenter}[2]{B^{#1}({#2})}
\newcommand{\BallRadius}[1]{B_{#1}}
\newcommand{\BallVolume}[1]{B^{#1}}
\newcommand{\cn}{c_{n,s}}
\newcommand{\omegan}{\omega_{n}}
\newcommand{\rzero}{r_{0}}

\newcommand{\BR}{\BallRadius{R}}

\newcommand{\Br}{\BallRadius{r}}
\newcommand{\Brx}{\BallRadiusCenter{r}{x}}

\newcommand{\Brhoy}{\BallRadiusCenter{\rho}{y}}

\newcommand{\BU}{\BallVolume{|U|}}

\newcommand{\BUy}{\BallVolumeCenter{|U|}{y}}
\newcommand{\PsiLambda}[1]{\sEnergyGeometric{\left(#1;\BallRadiusCenter{r}{x},\Lambda}\right)}

\newcommand{\integraldex}[2]{\integral{#1}#2\integralde x}

\newcommand{\fracTotalVar}[2]{[D #2]_{#1}}
\newcommand{\sTotalVar}[1]{\fracTotalVar{s}{#1}}
\newcommand{\EnergyLOneFracTV}[2]{\mathcal{E}^{#2}#1}
\newcommand{\EnergyGeometric}[2]{\mathcal{E}^{#2}_{\mathbf{G}}#1}
\newcommand{\sEnergyLOneFracTV}[1]{\EnergyLOneFracTV{#1}{s}}
\newcommand{\sEnergyGeometric}[1]{\EnergyGeometric{#1}{s}}

\DeclareMathOperator*{\supp}{supp}
\newcommand{\Holder}{H\"{o}lder}
\DeclareMathOperator*{\esssup}{ess-sup}
\DeclareMathOperator*{\essinf}{ess-inf}

\newcommand{\Uk}{U_{k}}

\newcommand{\Vk}{V_{k}}
\newcommand{\Un}{U_{n}}
\newcommand{\fk}{f_{k}}
\newcommand{\ueps}{u_{\eps}}
\newcommand{\feps}{f_{\eps}}
\newcommand{\interaction}{L_{n,s}}
\newcommand{\distSolDatPLUS}{\mu_{s}^{+}}
\newcommand{\distSolDatMINUS}{\mu_{s}^{-}}
\newcommand{\distSolDatPLUSMINUS}{\mu_{s}^{\pm}}

\newcommand{\sTVLone}{sTV\hspace*{-0.1 cm}-\hspace*{-0.1 cm}L^1}
\newcommand{\TVLone}{TV\hspace*{-0.1 cm}-\hspace*{-0.1 cm}L^1}
\newcommand{\TVLtwo}{TV\hspace*{-0.1 cm}-\hspace*{-0.1 cm}L^2}
\newcommand{\Esedoglu}{Esedo\={g}lu }
\newcommand{\Et}{E_t}
\newcommand{\Ut}{U_t}
\newcommand{\Ezero}{E_0}
\newcommand{\Uzero}{U_0}
\newcommand{\U}{U}

\newcommand{\ULambdaone}{U_{\Lambda}^{1}}
\newcommand{\ULambdatwo}{U_{\Lambda}^{2}}
\newcommand{\ULambdai}{U_{\Lambda}^{i}}
\newcommand{\ULambdaj}{U_{\Lambda}^{j}}
\newcommand{\Uone}{U_{1}}
\newcommand{\Utwo}{U_{2}}
\newcommand{\xzero}{x_{0}}

\newcommand{\Hnmenouno}{\mathcal{H}^{n-1}}
\newcommand{\Cheegs}{h_{s}}
\newcommand{\sCheegerClass}{\mathcal{C}_s}

\title{Fractional total variation denoising model with $L^1$ fidelity}

\subjclass[2020]{49Q20, 94A08.}
\keywords{Fractional total variation, Fractional perimeter, Image denoising, Regularity of minimizers.}

\author{K. Bessas}
\address{University of Pisa, Department of Mathematics,
	Largo Bruno Pontecorvo, 5, 
	56127 Pisa PI, Italy}
\email{konstantinos.bessas@phd.unipi.it}

\begin{document}
\begin{abstract}
	We study a nonlocal version of the total variation-based model with $L^1$ fidelity for image denoising, where the regularizing term is replaced with the fractional $s$-total variation. We discuss regularity of the level sets and uniqueness of solutions, both for high and low values of the fidelity parameter. 
	We analyse in detail the case of binary data given by the characteristic functions of convex sets.
\end{abstract}

\maketitle

\keywords

\tableofcontents

\section{Introduction}\label{sec:introduction}

Let $f\in\Lspace{1}(\Renne)$ a given function, $\Lambda>0$ and $s\in(0,1)$ two parameters. We are interested in studying the following variational problem:
\begin{equation}\label{eq:OurModel}
\min_{u\in\Wspace{s}{1}(\Renne)}\unmezzo\integral{\Renne}\integral{\Renne}\frac{|u(x)-u(y)|}{|x-y|^{n+s}}\integralde x\integralde y+\Lambda\integraldex{\Renne}{|u-f|},\tag{$\sTVLone$}
\end{equation}
where $\Wspace{s}{1}(\Renne)$ is the fractional Sobolev space of order $(s,1)$.

A motivation for considering such an energy comes from image processing.
In particular, \eqref{eq:OurModel} can be interpreted as a denoising model. In this setting the dimension is equal to two, even though the results that we will present hold for any positive integer $n$. We consider an image on a certain screen, coinciding with the whole $\Renne$ in our case, which is the distorted version of a clearer initial image. We treat greyscale images only, and each one of them is represented by a real function over $\Renne$, namely its greyscale. Therefore, with a little abuse of notation, we will identify an image with its greyscale in the sequel. In this problem, $f$ is interpreted as the given image, which may be somehow degraded by some ``\textit{noise}'', while any solution $u$ of \eqref{eq:OurModel} is interpreted as a \textit{denoised} $f$, so that the noise is captured by their difference $f-u$. 
According to the model, on the one hand $u$ aims at approximating $f$, depending on the value of the \textit{fidelity} parameter $\Lambda$, because of the second addend of the functional, which is an $\Lspace{1}$-norm.
On the other hand $u$ aims at regularizing $f$, thanks to the first addend of the functional, which is the \textit{fractional $s$-total variation} of $u$, or -equivalently- half the seminorm of $\Wspace{s}{1}(\Renne)$.

We will now explain why the model \eqref{eq:OurModel} takes actually this form starting from a review of its more classical versions. 

A motivation for our study is the total variation-based image denoising model deeply investigated by Chan and \Esedoglu in \cite{ChaEse}:

\begin{equation}\label{eq:ChanEsedogluModel}
\min_{u\in\BoundedVariation(\Renne)}\integraldex{\Renne}{|\nabla u|}+\Lambda\integraldex{\Renne}{|u-f|},\tag{$\TVLone$}
\end{equation}
where $f\in\Lspace{1}(\Renne)$, $\Lambda>0$ and $\BoundedVariation(\Renne)$ is the space of functions of bounded variation.
\eqref{eq:ChanEsedogluModel} arises as a variant of the total variation-based
model of Rudin, Osher and Fatemi (also referred as ROF model) introduced in \cite{RudOshFat}:
\begin{equation}\label{eq:ROFModel}
\min_{u\in\BoundedVariation(\Renne)}\integraldex{\Renne}{|\nabla u|}+\Lambda\integraldex{\Renne}{|u-f|^2},\tag{$\TVLtwo$}
\end{equation}
with $f\in\Lspace{2}(\Renne)$ and $\Lambda>0$.

The difference between \eqref{eq:ROFModel} and \eqref{eq:ChanEsedogluModel} lies in the approximation term, since in the former case the kind of approximation is in $\Lspace{2}$, while in the latter it is in $\Lspace{1}$.
From a mathematical point of view, with this modification the original ROF functional loses its strict convexity, becoming merely convex. This is immediately reflected in the lack of uniqueness of solutions of \eqref{eq:ChanEsedogluModel}, which has interesting consequences in applications. Furthermore, the new functional is $1$-homogeneous, which makes the model \eqref{eq:ChanEsedogluModel} contrast invariant (cf. \Cref{rem:contrastInvariance}), differently from \eqref{eq:ROFModel}. 
On the whole, from a numerical point of view, the model \eqref{eq:ChanEsedogluModel} can overcome several difficulties arising from the model \eqref{eq:ROFModel} because of its sensitivity to the geometric character of the features of the given image, i.e. their shapes, rather than the contrast among them. 
For a deeper comparison between these two models and detailed explanations of numerical results we refer to \cite{ChaEse} and the references therein.
 
The models presented above are local, in the sense that the regularizing term relies on a local operator, namely the total variation of functions of bounded variation.   

As explained in \cite{BuaColMor}, the local methods, such as the ones based on the total variation, aim at a noise reduction and at a reconstruction of the main geometrical configurations but not at the preservation of the fine structure, details, and texture of the given image. Due to the weak regularity assumptions on the given image in the previous models, details
and fine structures are smoothed out because they are treated as noise.
Therefore, considering these issues, Buades, Coll and Morel turned their attention towards nonlocal approaches.  

Later, in \cite{GilOsh07}, Gilboa and Osher proposed a generalization of the total variation in the nonlocal framework \cite[see][equation (C.1)]{GilOsh07}:

\begin{equation}\label{eq:wtotvar}
J_1(u):=\frac{1}{2}\integral{\Omega\times\Omega}|u(x)-u(y)| w(x,y)\integralde x \integralde y,
\end{equation}
where $\Omega$ is a domain of $\Renne$ and $w$ is a \textit{weight} function, i.e. a positive function defined in $\Omega\times\Omega$ satisfying the following symmetry property: $w(x,y) = w(y,x)$ for all $x,y\in\Omega$.
So, \eqref{eq:wtotvar} is interpreted as a weighted nonlocal total variation.

Then, the authors in \cite{GilOsh08} proposed several generalizations of this object  \cite[see][equations (2.10), (2.11)]{GilOsh08}, introducing for instance a kind of \textit{isotropic} nonlocal total variation \cite[equation (2.14)]{GilOsh08} as well as its \textit{anisotropic} counterpart \cite[equation (2.15)]{GilOsh08}.

In both works, denoising models made of a nonlocal total variation term and an approximation term, such as approximations involving $\Lspace{1}$ and $\Lspace{2}$ distances, are analysed and numerically applied.
In particular, nonlocal versions of \eqref{eq:ROFModel} are investigated 
\cite[cf.][equation (2.5),(2.7)]{GilOsh07}, \cite[cf.][equation (4.1)]{GilOsh08}. It is also shown that the nonlocal version of \eqref{eq:ChanEsedogluModel} \cite[see][equation (4.3)]{GilOsh08} can be used to detect and remove irregularities from textures.

Moreover, in \cite{MazSolTol}, Mazón, Solera and Toledo analysed the $(BV,L^1)$ and the $(BV,L^2)$ minimization problems in metric random walk spaces $[X,d,m]$, which, in particular cases \cite[see][Example 1.1]{MazSolTol}, translate themselves into denoising models made of an approximation term in $L^1$ (respectively $L^2$) and a nonlocal total variation term of the form:
\begin{equation}\label{eq:MazSolTolJTotVar}
	TV_{m^J}(u):=\frac{1}{2}\integral{\Renne\times\Renne}|u(x)-u(y)| J(x-y)\integralde x \integralde y,
\end{equation}
where $J\in\Lspace{1}(\Renne)$ satisfies the condition $\integraldex{\Renne}{J(x)}=1$. For the definitions and main properties of nonlocal objects arisen from nonsingular kernels like $J$ we refer to the monograph \cite{MazRosTol}, although in our work we will only focus our attention on the fractional kernel, which is singular.

The fractional $s$-total variation appearing in \eqref{eq:OurModel} is very similar -at least formally- to the nonlocal functional \eqref{eq:wtotvar} for a suitable choice of $w$ and to \eqref{eq:MazSolTolJTotVar}, upon replacing $J$ with the fractional kernel. However, the authors in \cite{GilOsh07}, \cite{GilOsh08} and \cite{MazSolTol} considered different type of weights and kernels. 

Nevertheless, our choice is naturally motivated by the following \textit{coarea formula} (see also \eqref{eq:sPerCoarea}):
\begin{align}\label{eq:sPerCoareaINTRO}
\unmezzo\integral{\Renne}\integral{\Renne}\frac{|u(x)-u(y)|}{|x-y|^{n+s}}\integralde x\integralde y=\int_{-\infty}^{+\infty}\Ps(\{u>t\})\integralde t,
\end{align}
where $\Ps$ is defined for all measurable sets $E\subset\Renne$ as
\begin{align}\label{eq:sFracPerDef}
\Ps(E):=\integral{E}\integral{\comp{E}}\frac{1}{|x-y|^{n+s}}\integralde x\de y,
\end{align}
and it is called the $s$-fractional perimeter. This type of perimeter has gained more and more attention since the paper \cite{CafRoqSav}, where Caffarelli, Roquejoffre and Savin studied it in the context of nonlocal minimal surfaces.

\eqref{eq:sPerCoareaINTRO} is the fractional analogue of the \textit{coarea formula} for the classical De Giorgi's perimeter \cite[see][Theorem 3.40]{AmbFusPal} and allows us to generalize to the model \eqref{eq:OurModel} many known techniques for dealing with the model \eqref{eq:ChanEsedogluModel}. 

It could be also worth remarking that in \cite{CinSerVal} Cinti, Serra and Valdinoci pointed out some advantages of using fractional perimeters instead of the classical one in image processing.

Furthermore, as explained in \cite{YCZC}, currently, fractional calculus provides an important tool for image denoising, even though the models cited by the authors of \cite{YCZC} differ from the model \eqref{eq:OurModel}, being based on other fractional objects. Specifically, many fractional model-based methods which were proposed make use of fractional-order derivatives to replace the regularizing term in \eqref{eq:ROFModel}. It was shown that the fractional-order derivative not only maintains the contour feature in the smooth area of the image, but also preserves high-frequency components like edges and textures.

Finally, in \cite{NovOno21} the authors considered the model corresponding to \eqref{eq:OurModel} with $L^2$-fidelity.

\smallskip

One of the key results of our work is the following: 
\begin{thm}[cf. \Cref{thm:lemmetto per lavoro su variazione seconda}]
	Let $\measurablesets$ denote the class of Lebesgue measurable sets of $\Renne$ and let $s\in(0,1)$ and $E\subseteq\Renne$ with $\partial E\in C^{1,1}$ either bounded or with bounded complement.
	Then, there exists $\Lambda(E,s)>0$ s.t. 
	\begin{equation*}
		\Ps(E)-\Ps(\U)\leq\Lambda|E\Delta \U|
	\end{equation*}
	for every $\U\in\measurablesets$ and $\Lambda\geq\Lambda(E,s)$.
\end{thm}
This is the fractional version of a result for the classical perimeter \cite[see for instance][Lemma 4.1]{AceFusMor} and might be of independent interest. In fact, this may be interpreted as a uniform estimate of the difference between the $s$-perimeter of a given regular set and the one of any measurable set, w.r.t. their $\Lspace{1}$ distance. 
While in the classical case this is a simple consequence of De Giorgi's Structure Theorem, in the nonlocal framework the result is more delicate. In fact, we first prove it when the set $E$ is a ball; then, through a comparison argument, we establish the result in the general case.

Instead the other results are more focused on the properties of solutions of \eqref{eq:OurModel}, such as the following one, which is about regularity of superlevel and sublevel sets of solutions:
\begin{thm}[cf. \Cref{cor:regularity}]
	Let $s\in(0,1)$, $\Lambda>0$ and $f\in\Lspace{1}(\Renne)$.
	If $u$ solves \eqref{eq:OurModel}, then all of its superlevel and sublevel sets 
	have boundary of class $C^1$ outside a closed singular set $S$ of Hausdorff
	dimension at most $n-3$.
\end{thm}
This is a consequence of regularity of local almost minimizers of
the fractional perimeter \cite[see for instance][Theorem 3.3]{CesNov17}.

The following two theorems are concerned about sufficient conditions for uniqueness of solutions of \eqref{eq:OurModel}. The first one tells that for high values of the fidelity parameter, the given datum is recovered as unique solution, provided that it is regular enough:
\begin{thm}[cf. \Cref{thm:highFidelity}]
	Let $s\in(0,1)$ and $f\in\Lspace{1}(\Renne)$ with superlevel sets uniformly bounded in $C^{1,1}$ (i.e., there exists $r>0$ such that every superlevel set of $f$ and its complement are the union of balls of radius $r$).
	Then, there exists $\Lambda(f,s)>0$ s.t. $f$ is the unique solution of \eqref{eq:OurModel} for every fidelity parameter $\Lambda\geq\Lambda(f,s)$.
\end{thm} 
The second one deals with low values of the fidelity parameter. For compactly supported data, the unique solution is the trivial one: 
\begin{thm}[cf. \Cref{thm:lowFidelity}]
	Let $s\in(0,1)$ and $R>0$ and $f\in\Lspace{1}(\Renne)$ such that $\supp(f)\subset\BR$. Then, there exists $\Lambda(R,n,s)>0$ such that $0$ is the unique solution of \eqref{eq:OurModel} for every fidelity parameter $0<\Lambda<\Lambda(R,n,s)$.
\end{thm}
The argument for proving \Cref{thm:lowFidelity} comes from adaptations of strategies presented in \cite{ChaEse}, while the one for \Cref{thm:highFidelity} is more technical, making use of \Cref{thm:lemmetto per lavoro su variazione seconda}.

The final results characterize more precisely solutions in the case of data which are characteristic functions of bounded and convex sets: 
\begin{thm}[cf. \Cref{thm:convexsCheeger}]
	Let $s\in(0,1)$ and $E\subset\Renne$ be a bounded convex set with nonempty interior. Let $\Cheegs(E)$ be the s-Cheeger constant of $E$ and $\sCheegerClass(E)$ be the class of s-Cheeger sets of $E$, together with the empty set.
	Then, \eqref{eq:OurModel}, with datum $f=\characteristic{E}$, admits a unique solution for $\lebone-$a.e. $\Lambda>0$. Furthermore,
	\begin{enumerate}[label=(\roman*)]
		\item\label{item:one} if $0<\Lambda<\Cheegs(E)$, then $0$ is the unique solution;
		\item\label{item:two} if $\Lambda=\Cheegs(E)$, then every function $u\in\Wspace{s}{1}(\Renne;[0,1])$, such that $\{u>t\} \in \sCheegerClass(E)$ for all $t\in [0,1)$,  is a solution;
		\item\label{item:three} if $\Lambda>\Cheegs(E)$, then $0$ is not a solution.
	\end{enumerate}
\end{thm}
\begin{thm}[cf. \Cref{thm:convexsCalibrable}]
	Let $s\in(0,1)$ and $E\subset\Renne$ be a bounded convex set which is a s-Cheeger set for itself. 
	Then, $\characteristic{E}$ is the unique solution for \eqref{eq:OurModel}, with datum $f=\characteristic{E}$, for all fidelity parameters $\Lambda>\Cheegs(E)$.
\end{thm}

The statements and proofs of \Cref{thm:convexsCheeger} and \Cref{thm:convexsCalibrable} require the notion of \textit{s-Cheeger set} (cf. \Cref{def:sCheeger}). We remark that the authors in \cite{MazSolTol} provided results in a similar fashion, after introducing a suitable notion of Cheeger set  \cite[see][Corollary 3.23 and Corollary 3.25]{MazSolTol}.

Finally, as a consequence of \Cref{thm:convexsCalibrable} and \Cref{cor:regularity}, we prove:

\begin{thm}[cf. \Cref{cor:ConvexSCalibrableRegular}]
	Let $s\in(0,1)$ and $E\subset\Renne$ be a bounded convex set which is $s$-calibrable, i.e. it is a s-Cheeger set of itself. Then, $\partial E$ is of class $C^1$. 
\end{thm}

The plan of the article is the following.

In \Cref{sec:notation} we introduce the notation that we will use throughout the work and we recall some needed results, which are well-known in the literature.

In \Cref{sec:existenceProperties} we first study the existence of solutions for the model \eqref{eq:OurModel} together with its basic properties, such as stability results. 
Then, we translate the problem \eqref{eq:OurModel} into a geometric one, giving a representation of the energy of the model \eqref{eq:OurModel}, via \textit{coarea} and \textit{layer-cake} formulas, which allows to establish equivalence between the original problem and its geometric counterpart (cf. \Cref{prop:equivalenceFuncGeomPBs}). Subsequently, we prove regularity for the sublevel and superlevel sets of solutions. 

In \Cref{sec:lemmetto}, we start focusing on \eqref{eq:OurModel} in the binary case, i.e. when the datum of the problem is a characteristic function of some measurable set, emphasizing the case where the set is bounded and convex.
Then, we prove \Cref{thm:lemmetto per lavoro su variazione seconda} and later we study how the size of the \text{fidelity} parameter can influence the behaviour of solutions of \eqref{eq:OurModel}.
Finally, after a brief recall on \textit{s-Cheeger sets}, we go back to the binary datum case, highlighting further features of the convex case.

\section{Notation and preliminary results}\label{sec:notation}

Let $n$ be an integer strictly greater than zero. We let
\begin{align*}
&\lebn := n-\text{dimensional Lebesgue (external) measure};\\
&\measurablesets := \text{Lebesgue measurable sets of } \Renne;\\
&\Hnmenouno:= (n-1)-\text{dimensional Hausdorff (external) measure}.
\end{align*}

Let $E,F\in\measurablesets$, $s\in(0,1)$, $R,m>0$, $x\in\Renne$. Then,
\begin{align*}
&\comp{E}:=\Renne\setminus E;\\
&|E|:=\lebn(E);\\
&\partial E:= \lebn-\text{measure theoretic boundary of }E;\\
&\Ps(E) := s-\text{fractional perimeter  of } E \text{ }(\text{see } \eqref{eq:sFracPerDef});\\
&\Emini \text{ and } \Emaxi \text{ are defined in \Cref{defn:MaxMinSolutions}};\\
&\BallRadiusCenter{R}{x}:=\text{ open ball in } \Renne \text{ of radius } R \text{ centred in } x;\\
&\BallRadius{R}:=\BallRadiusCenter{R}{0};\\
&\BallVolumeCenter{m}{x}:=\text{open ball in } \Renne \text{ of volume } m \text{ centred in } x;\\
&\BallVolume{m}:=\BallVolumeCenter{m}{0};\\
&\cn:=\Ps(\BallRadius{1});\\
&\omegan:=|\BallRadius{1}|.
\end{align*}

We also write $E=F$ when $|E\Delta F|=0$; $E\subseteq F$  if $|E\setminus F|=0$; and $E\subsetneqq F$ when $|E\setminus F|=0$ and $|F\setminus E|>0$. Moreover, $E\sqcup F$ is used to denote $E\cup F$ when $E\cap F=\emptyset$.

If $U\in\measurablesets$, by $U\in \Lspace{1}(\Renne)$ we mean $\characteristic{U}\in\Lspace{1}(\Renne)$ . Similarly, if $\Uk\in\Lspace{1}(\Renne)$ for every $k\in\N$, we write  $\Uk\to U$ in $\Lspace{1}(\Renne)$  to denote $\characteristic{\Uk}\to \characteristic{U}$ in $\Lspace{1}(\Renne)$. We adopt the same convention also for $\Llocspace{1}(\Renne)$.

Let $\Omega$ be an open subset of $\Renne$ and $u:\Omega\to [-\infty,+\infty]$ a $\lebn$-measurable function. If $u\in\Wspace{s}{1}(\Omega)$, we denote the \textit{Gagliardo seminorm} in of $u$ by 
\begin{align*}
[u]_{\Wspace{s}{1}(\Omega)}:=\integral{\Omega}\integral{\Omega}\frac{|u(x)-u(y)|}{|x-y|^{n+s}}\integralde x\integralde y.
\end{align*}
For the definitions and properties of fractional Sobolev spaces $\Wspace{s}{p}$, we refer to \cite{DiNPalVal}.
If $u\in\Llocspace{1}(\Renne)$ we define the $s$-\textit{total variation} of $u$ over $\Renne$ by
\begin{align*}
\sTotalVar{u}:=\unmezzo\integral{\Renne}\integral{\Renne}\frac{|u(x)-u(y)|}{|x-y|^{n+s}}\integralde x\integralde y.
\end{align*}
Furthermore, we put
\begin{align*}
&u^+:=\max\{u,0\}, &u^-:=\max\{-u,0\}.
\end{align*}
If $t\in\R$, we let 
\begin{align*}
	\{u>t\}:=\Set{x\in\Renne}{u(x)>t},
\end{align*}
and we adopt an analogue convention for $\{u<t\},\{u\geq t\},\{u \leq t\}$. 

If a function or a set depend on parameters we sometimes omit one or more of them to avoid heavy notation, if they are clear from the context. 

By saying that a real function $f$ is \textit{binary}, we mean that it is the characteristic function of some $\lebn$-measurable set.

Let $E,F\in\measurablesets$, $s\in (0,1)$, $\lambda>0$ and $K\subseteq\Renne$ convex. Then,
\begin{align*}
&\Ps(E)=\sTotalVar{\characteristic{E}};\\
&\Ps(E)=\Ps(\comp{E});\\
&\Ps(\lambda E)=\lambda^{n-s}\Ps(E) \textit{ (Scaling)};\\
&\Ps(E\cap F)+\Ps(E\cup F)\leq\Ps(E)+\Ps(F)\textit{ (Submodularity)};\\
&\Ps(E)<+\infty, \text{ if } |E|<+\infty \text{ and } \Hnmenouno(\partial E)<+\infty;\\
&\Ps(E\cap K)\leq\Ps(E), \text{ if } |E|<+\infty.
\end{align*}
Some of the previous properties are straightforward consequences of the definitions. The proofs of the last three of them can be found in \cite[Proposition 2.2, Remark 2.1]{CesNov} and \cite[Lemma B.1]{FigFusMagMilMor}. 

The following \textit{coarea formula} is proved (even for more general kernels) in \cite[Proposition 2.3]{CesNov}:
\begin{align}\label{eq:sPerCoarea}
\sTotalVar{u}=\int_{-\infty}^{+\infty}\Ps(\{u>t\})\integralde t,
\end{align}
with $u\in\Llocspace{1}(\Renne)$.

We also recall that $u\mapsto\sTotalVar{u}$ is a convex and lower semicontinuous functional in $\Llocspace{1}(\Renne)$.

We define the measure $\interaction$ on $\Rtwon$ by
\begin{align}
\label{eq:interaction measure}
\de\interaction(x,y):=\frac{1}{|x-y|^{n+s}}\de\leb^{2n}(x,y),
\end{align}
so that $\interaction(E\times\comp{E})=\interaction(\comp{E}\times E)=\Ps(E)$.
We also note that $\interaction(E\times F)=\interaction(F\times E)$, $\interaction<<\leb^{2n}$ and $\leb^{2n}<<\interaction$, i.e. each measure is absolutely continuous with respect to the other.

\section{Existence and regularity of minimizers}\label{sec:existenceProperties}

We first show existence of solutions for the model \eqref{eq:OurModel}.
We give some preliminary definitions.
\begin{defn}\label{def:sEnergy}
	Let $\Lambda>0$, $s\in(0,1)$ and $f\in\Llocspace{1}(\Renne)$. We define $\sEnergyLOneFracTV=\sEnergyLOneFracTV(\cdot;f,\Lambda):\Llocspace{1}(\Renne)\to[0,+\infty]$, the functional (or energy) associated to the problem $\eqref{eq:OurModel}$ by
	\begin{align*}
	\sEnergyLOneFracTV{(u)}=\sEnergyLOneFracTV{(u;f,\Lambda)}:=\sTotalVar{u}+\Lambda\int_{\Renne}|u-f|\integralde x.
	\end{align*}
	Furthermore, 
	\ref{Pb:FunctionalProblem} refers to the following  minimum problem:
	\begin{align}\label{Pb:FunctionalProblem}\tag*{$\FunctionalProblem{f}{\Lambda}$}
	\min_{u\in\Llocspace{1}(\Renne)}\sEnergyLOneFracTV{(u;f,\Lambda)}.
	\end{align}
\end{defn}
\begin{remark}
	The problem \ref{Pb:FunctionalProblem} is convex, but not strictly convex.
\end{remark}
\begin{defn}
	Let $\Lambda>0$, $s\in(0,1)$ and $f\in\Llocspace{1}(\Renne)$.
	We say that $u\in\Llocspace{1}(\Renne)$ is a solution of \ref{Pb:FunctionalProblem} if and only if:
	\begin{enumerate}
		\item $\sEnergyLOneFracTV{(u;f,\Lambda)}<+\infty$;
		\item $\sEnergyLOneFracTV{(u;f,\Lambda)}\leq \sEnergyLOneFracTV{(w;f,\Lambda)}$ for all $w\in\Llocspace{1}(\Renne)$.
	\end{enumerate}
	We denote the (possibly empty) set of solutions of \ref{Pb:FunctionalProblem} by $\FunctionalProblemSolutions{f,\Lambda}$.
\end{defn}

We are now ready to prove the first existence result, that is existence for a datum in $\Lspace{1}(\Renne)$.
\begin{prop}\label{prop:existenceL1Datum}Let $s\in(0,1)$.
	For every $\Lambda>0$ and $f\in\Lspace{1}(\Renne)$, the problem
	\ref{Pb:FunctionalProblem} admits solutions and they belong to $\Wspace{s}{1}(\Renne)$.
\end{prop}
\begin{proof}
	By Fatou's Lemma, the functional $\sEnergyLOneFracTV$ is $\Llocspace{1}(\Renne)$-sequentially lower semicontinuous. Moreover, its sequential compactness comes from the compact Sobolev embedding of $\Wspace{s}{1}(\Renne)$ into $\Llocspace{1}(\Renne)$ \cite[It is a consequence of][Theorem 7.1]{DiNPalVal}. This leads to the existence of a solution in $\Llocspace{1}(\Renne)$, which is easily seen to belong to $\Wspace{s}{1}(\Renne)$, thanks to the integrability of the datum $f$.
\end{proof}

\begin{remark}\label{rem:ConvexityFunctionalProblemSolutions}
	In view of \Cref{prop:existenceL1Datum}$, \emptyset\neq\FunctionalProblemSolutions{f,\Lambda}\subset \Wspace{s}{1}(\Renne)$ if $s\in(0,1)$, $f\in\Lspace{1}(\Renne)$ and $\Lambda>0$.
	Furthermore, the convexity of the energy $\sEnergyLOneFracTV{}$ makes the set $\FunctionalProblemSolutions{f,\Lambda}$ convex, and in general its cardinality could be strictly greater than one (see \Cref{thm:convexsCheeger}). Since $\sEnergyLOneFracTV{}$ is $\Llocspace{1}(\Renne)$-lower semicontinuous, $\FunctionalProblemSolutions{f,\Lambda}$ is closed in $\Llocspace{1}(\Renne)$.   
\end{remark}

Now we start highlighting some properties of the problem \ref{Pb:FunctionalProblem}.
The first one is a stability result which shows the behaviour of solutions of a sequence of problems with data converging to a certain function in $\Lspace{1}(\Renne)$.

\begin{prop}[Stability 1]\label{prop:stabilityP}
	Let $s\in(0,1)$, $\Lambda>0$, $\sequence{\fk}{k}\subset\Lspace{1}(\Renne)$, $f\in\Lspace{1}(\Renne)$, $u\in\Llocspace{1}(\Renne)$ and $\uk\in\FunctionalProblemSolutions{\fk,\Lambda}$ for every $k\in\N$.
	If $\uk\to u$ in $\Llocspace{1}(\Renne)$ and $\fk\to f$ in $\Lspace{1}(\Renne)$, then $u\in\FunctionalProblemSolutions{f,\Lambda}\subset \Wspace{s}{1}(\Renne)$.
\end{prop}
\begin{proof}
	Fix $v\in\Wspace{s}{1}(\Renne)$ and $k\in\N$. By minimality of $\uk$, we can write
	\begin{align*}
	\sTotalVar{\uk}+\Lambda\int_{\Renne}|\uk-\fk|\integralde x\leq \sTotalVar{v}+\Lambda\int_{\Renne}|v-\fk|\integralde x.
	\end{align*}
	Since the fractional total variation is lower semicontinuous with respect to the $\Llocspace{1}(\Renne)$ topology, by letting $k\to+\infty$ we get
	\begin{align*}
	\sTotalVar{u}+\Lambda\int_{\Renne}|u-f|\integralde x\leq \sTotalVar{v}+\Lambda\int_{\Renne}|v-f|\integralde x.
	\end{align*}
\end{proof}

The following proposition allows us to describe the set of solutions of a class of problems \ref{Pb:FunctionalProblem} whose datum is a combination or a modification of other data.

\begin{prop}[Properties of solutions]\label{prop:basicPropertiesFuncProb}
	Let $s\in(0,1)$, $\Lambda>0$, $c\in\R$, $g\in\Llocspace{1}(\Renne)$.
	\begin{enumerate}[label=(\roman*)]
		\item\label{item:translation} $\FunctionalProblemSolutions{g+c,\Lambda}=\FunctionalProblemSolutions{g,\Lambda}+c$;
		\item\label{item:dilation} $c\hspace*{0.02cm}\FunctionalProblemSolutions{g,\Lambda}=\FunctionalProblemSolutions{c g,\Lambda}$; 
		\item\label{item:truncation} if $u\in\FunctionalProblemSolutions{g,\Lambda}$, then $u^{\pm}\in\FunctionalProblemSolutions{g^{\pm},\Lambda}$;
		\item\label{item:truncationbis} if $u\in\FunctionalProblemSolutions{g,\Lambda}$, then $\min\{u,c\}\in\FunctionalProblemSolutions{\min\{g,c\},\Lambda}$ and\\ $\max\{u,c\}\in\FunctionalProblemSolutions{\max\{g,c\},\Lambda}$.
	\end{enumerate}
\end{prop}
\begin{proof}
	The argument that we will use in the proof is an adaptation of \cite[Lemma 3.5]{NovPao}.
	We shall prove only $\ref{item:truncation}$ and $\ref{item:truncationbis}$, since $\ref{item:translation}$ and $\ref{item:dilation}$ are immediate consequences of the definitions. A direct computation shows that
	\begin{align*}
	\sTotalVar{u}=\sTotalVar{u^+}+\sTotalVar{u^-};
	\end{align*}
	\begin{align*}
	|g-u|=|g^+-u^+|+|g^--u^-|.
	\end{align*}
	Now we fix a perturbation $\psi\in\Llocspace{1}(\Renne)$ and, being $u$ a solution of $\FunctionalProblem{g}{\Lambda}$, we write:
	\begin{align*}
	&\Lambda\int_{\Renne}|g^+-u^+|\integralde x +\sTotalVar{u^+} +\Lambda\int_{\Renne}|g^--u^-|\integralde x +\sTotalVar{u^-}\\
	&=\Lambda\int_{\Renne}|g-u|\integralde x +\sTotalVar{u}\leq\Lambda\int_{\Renne}|u+\psi-g|\integralde x +\sTotalVar{(u+\psi)}\\
	&\leq\Lambda\int_{\Renne}|u^++\psi-g^+|\integralde x +\sTotalVar{(u^++\psi)}+\Lambda\int_{\Renne}|u^--g^-|\integralde x +\sTotalVar{u^-}.
	\end{align*}
	From this chain of inequalities we deduce $u^{+}\in\FunctionalProblemSolutions{g^{+},\Lambda}$; while with slightly different estimates, it can be proved that $u^{-}\in\FunctionalProblemSolutions{g^{-},\Lambda}$. This implies $\ref{item:truncation}$. 
	
	$\ref{item:truncationbis}$ follows from the identities $\min\{u,c\}=c-(u-c)^-$, $\max\{u,c\}=(u-c)^++c$ and $\ref{item:translation}, \ref{item:dilation},\ref{item:truncation}$.
\end{proof}
\begin{remark}\label{rem:contrastInvariance}
	$\ref{item:dilation}$ is interpreted as \textit{contrast invariance} of the model \eqref{eq:OurModel} in the image processing setting (see also \Cref{sec:introduction})
\end{remark}

We now define a functional which can be thought of as the geometric counterpart of the one defined in \Cref{def:sEnergy}. 
\begin{defn}\label{def:GeometricsEnergy}
	Let $\Lambda>0$, $s\in(0,1)$ and $E\in\measurablesets$. We define the functional (or energy) $\sEnergyGeometric=\sEnergyGeometric(\cdot;E,\Lambda):\measurablesets\to[0,+\infty]$ by
	\begin{align*}
	\sEnergyGeometric{(U)}=\sEnergyGeometric{(U;E,\Lambda)}:=\Ps(U)+\Lambda|E\Delta U| .
	\end{align*}
	
\end{defn}

Thanks to \eqref{eq:sPerCoarea}, we are led to the following geometric representation of $\sEnergyLOneFracTV$:
\begin{prop}
	For every $s\in(0,1)$, $\Lambda>0$, $u\in\Wspace{s}{1}(\Renne)$ and $f\in\Lspace{1}(\Renne)$ the following identity holds:
	\begin{align}\label{eq:LinkFuncGeom}
	\sEnergyLOneFracTV{(u;f,\Lambda)}
	=\int_{-\infty}^{+\infty}\sEnergyGeometric{(\{u>t\};\{f>t\},\Lambda)}\integralde t.
	\end{align}
\end{prop}
\begin{proof}
	Apply \eqref{eq:sPerCoarea} and the \textit{layer-cake} formula to $\sEnergyLOneFracTV$ \cite[see][Proof of Proposition 5.1]{ChaEse}.
\end{proof}

We will now relate \ref{Pb:FunctionalProblem} and the variational problem corresponding to the geometric functional $\sEnergyGeometric{}$, which is introduced in the next definition.
\begin{defn}\label{def:GeometricsEnergySolutions}
	Let $s\in(0,1)$, $\Lambda>0$ and $E\in\measurablesets$.
	\ref{Pb:GeometricProblem} refers to the following (nonconvex) geometric minimum problem:
	\begin{align}\label{Pb:GeometricProblem}\tag*{$\GeometricProblem{E}{\Lambda}$}
	\min_{U\in\measurablesets}\sEnergyGeometric{(U;E,\Lambda)}=\min_{U\in\measurablesets}\Ps(U)+\Lambda|E\Delta U|.
	\end{align}
	We say that $U\in\measurablesets$ is a solution, or a \textit{(global) minimum}, of \ref{Pb:GeometricProblem} if and only if:
	\begin{enumerate}
		\item $\sEnergyGeometric{(U;E,\Lambda)}<+\infty$;
		\item $\sEnergyGeometric{(U;E,\Lambda)}\leq \sEnergyGeometric{(W;E,\Lambda)}$ for all $W\in\measurablesets$.
	\end{enumerate}
	We denote the (possibly empty) set of solutions of \ref{Pb:GeometricProblem} by $\GeometricProblemSolutions{E,\Lambda}$.
	
	We say that $U\in\measurablesets$ is a \textit{local minimum} of \ref{Pb:GeometricProblem} if and only if:
	there exists $\delta>0$ such that for every $\xzero\in\Renne$, for every $0<r<\delta$ and for every $W\in\measurablesets$ with $U\Delta W\subseteq\BallRadiusCenter{r}{\xzero}$ it holds:
	\begin{align*}
	\sEnergyGeometric{(U;E,\Lambda)}\leq \sEnergyGeometric{(W;E,\Lambda)}.
	\end{align*}

\end{defn}

\begin{remark}[Solutions with complement datum]
	\label{rem:Solutions with complement datum}
	Let $s\in(0,1)$, $\Lambda>0$ and $E\in\measurablesets$. Then, \ref{Pb:GeometricProblem} admits solution if and only if $\GeometricProblem{\comp{E}}{\Lambda}$ admits solution. Furthermore,
	\begin{equation}\label{eq:solutions with complement datum}
		\GeometricProblemSolutions{\comp{E},\Lambda}=\Set{\comp{\U}}{\U\in\GeometricProblemSolutions{E, \Lambda}},
	\end{equation}
	 even though the former sets might be empty.
	 
	To prove \eqref{eq:solutions with complement datum}, we first note that $\Ps(\U)=\Ps(\comp{\U})$ and $|E\Delta \U|=|\comp{E}\Delta \comp{\U}|$; so, \ref{Pb:GeometricProblem} is equivalent to
	\begin{align}
		\label{eq:EquivPBcomplement}
		\min_{\U\in\measurablesets}\Ps(\comp{\U})+\Lambda|\comp{E}\Delta \comp{\U}|.
	\end{align}
	
	Finally, we observe that the set of solutions of $\GeometricProblem{\comp{E}}{\Lambda}$ is given by
	\begin{equation*}
		\Set{\comp{\U}}{\U\spazio\text{is a solution of}\spazio\eqref{eq:EquivPBcomplement}},
	\end{equation*}
which shows \eqref{eq:solutions with complement datum}.
\end{remark}

The following proposition is the analogue of  \Cref{prop:stabilityP} for the geometric problem \ref{Pb:GeometricProblem}.

\begin{prop}[Stability 2]\label{prop:stabilityGP}
	Let $s\in(0,1)$, $\Lambda>0$, $\sequence{\Uk}{k},\sequence{\Fk}{k}\subset\measurablesets$ and $U,F\in\measurablesets$ such that $\Uk\in\GeometricProblemSolutions{\Fk,\Lambda}$ for every $k\in\N$.
	If $\Uk\to U$ in $\Llocspace{1}(\Renne)$ and $\Fk\to F$ in $\Lspace{1}(\Renne)$, then $U\in\GeometricProblemSolutions{F,\Lambda}$.
\end{prop}
\begin{proof}
	Let $V\in\measurablesets$ and $k\in\N$. By minimality of $\Uk$, we get
	\begin{align*}
	\Ps(\Uk)+\Lambda|\Uk\Delta\Fk|\leq\Ps(V)+\Lambda|V\Delta\Fk|.
	\end{align*}	
	Since $\Ps$ is lower semicontinuous with respect to the $\Llocspace{1}(\Renne)$ topology, by letting $k\to+\infty$ we get
	\begin{align*}
	\Ps(U)+\Lambda|U\Delta F|\leq\Ps(V)+\Lambda|V\Delta F|.
	\end{align*}
\end{proof}

The following result is one the key steps for the analysis of the model \eqref{eq:OurModel}.
It establishes a link between \ref{Pb:FunctionalProblem} and a suitable family of geometric problems $\sequence{\GeometricProblem{\Et}{\Lambda}}{t\in\R}$. 

\begin{prop}[Equivalence]\label{prop:equivalenceFuncGeomPBs}
	Let $s\in(0,1)$, $u\in\Wspace{s}{1}(\Renne)$, $f\in\Lspace{1}(\Renne)$ and $\Lambda>0$.
	Let $(\sequence{\Ut}{t\in\R},\sequence{\Et}{t\in\R})$ be equal to one of the following pair of sequences of superlevel or sublevel sets:
	\begin{align*}
	&(\sequence{\{u>t\}}{t\in\R},\sequence{\{f>t\}}{t\in\R});
	&(\sequence{\{u<t\}}{t\in\R},\sequence{\{f<t\}}{t\in\R});\\
	&(\sequence{\{u\geq t\}}{t\in\R},\sequence{\{f\geq t\}}{t\in\R});
	&(\sequence{\{u\leq t\}}{t\in\R},\sequence{\{f\leq t\}}{t\in\R}).
	\end{align*}
	Then,
	\begin{enumerate}[label=(\roman*)]
		\item\label{item:fromGPtoP} If $\Ut\in\GeometricProblemSolutions{\Et,\Lambda}$ for $\lebone$-a.e. $t:\essinf u\wedge\essinf f<t<\esssup u\vee\esssup f $, then $u\in\FunctionalProblemSolutions{f,\Lambda}$,
		\item\label{item:fromPtoGP}  If	$u\in\FunctionalProblemSolutions{f,\Lambda}$, then $\Ut\in\GeometricProblemSolutions{\Et,\Lambda}$ for all $t\in\R\setminus\{0\}$.
	\end{enumerate}
\end{prop}
\begin{remark}
	If $f\in\Lspace{1}(\Renne)$, then for every $t>0$, $|\{f>t\}|<+\infty$ and $|\{f\leq -t\}|<\infty$.
	So if $t\in\R\setminus\{0\}$, either $\{f>t\}$ or its complement has finite measure. Instead, both $\{f>0\}$ and its complement may have infinite measure.
\end{remark}

\begin{remark}\label{rem:zero super level is a local minimum}
	Using the same notation as in the statement of \Cref{prop:equivalenceFuncGeomPBs}, we observe that the case $t=0$ is more delicate. We can always prove $\Uzero$ is a local minimum for $\GeometricProblem{\Ezero}{\Lambda}$ (see \Cref{def:GeometricsEnergySolutions}). We can also prove it is a global minimum provided that we know either $\{f>0\}$ or its complement has finite measure, applying \Cref{prop:stabilityGP} and \Cref{prop:equivalenceFuncGeomPBs}.
\end{remark}

\begin{proof}[Proof of \Cref{prop:equivalenceFuncGeomPBs}]
	\newcommand{\A}{\GeometricProblemSolutions{\{f>t\},\Lambda}}
	\newcommand{\gammat}{\gamma_t}
	\newcommand{\gammatau}{\gamma_{\tau}}
	We start by proving $\ref{item:fromGPtoP}$.
	
	If $u\in\Wspace{s}{1}(\Renne)$ and its superlevel set $\{u>t\}$ is a solution of $\GeometricProblem{\{f>t\}}{\Lambda}$ for $\lebone$-a.e. $t:\essinf u\wedge\essinf f<t<\esssup u\vee\esssup f$, then $u$ solves \ref{Pb:FunctionalProblem} by a direct application of \eqref{eq:LinkFuncGeom}. The other cases are similar.
	
	At this point, we prove $\ref{item:fromPtoGP}$, adopting the strategy in \cite[Theorem 3.6]{NovPao}.
	
	Let $u\in\FunctionalProblemSolutions{f,\Lambda}$, $\eps>0$ and $t\in\R$. We define $\ueps(x):=(u(x)-t)/\eps\wedge 1\vee 0$ and $\feps(x):=(f(x)-t)/\eps\wedge 1\vee 0$ for all $x\in\Renne$. We observe that $\ueps$ and $\feps$ converge pointwise respectively to $\characteristic{\{u>t\}}$ and to $\characteristic{\{f>t\}}$. Since $\ueps=\ueps\characteristic{\{u>t\}}$ and  $\feps=\feps\characteristic{\{f>t\}}$, in both cases the convergence is also in $\Lspace{1}(\Renne)$ if $t>0$, by Lebesgue convergence theorem. Consequently, by \Cref{prop:basicPropertiesFuncProb} and \Cref{prop:stabilityP}, we deduce $\characteristic{\{u>t\}}\in\FunctionalProblemSolutions{\characteristic{\{f>t\}},\Lambda}$; so, $\{u>t\}\in\GeometricProblemSolutions{\{f>t\},\Lambda}$ for all $t>0$.
	If $t<0$, arguing similarly, we observe that $1-\ueps$ and $1-\feps$ converge respectively to  $\characteristic{\{u\leq t\}}$ and to $\characteristic{\{f\leq t\}}$ in $\Lspace{1}(\Renne)$. So, by \Cref{prop:basicPropertiesFuncProb} and \Cref{prop:stabilityP} again, we deduce $\characteristic{\{u\leq t\}}\in\FunctionalProblemSolutions{\characteristic{\{f\leq t\}},\Lambda}$. Then, $\{u\leq t\}\in\GeometricProblemSolutions{\{f\leq t\},\Lambda}$ for all $t<0$.
	
	Let $t>0$. We can find an increasing sequence $\sequence{\tau_k}{k}\subseteq (0,t)$  such that $\{u\geq t\}=\intersection[k\in\N]\{u>\tau_k\}$, $\{f\geq t\}=\intersection[k\in\N]\{f>\tau_k\}$, $\{u>\tau_k\}\in\GeometricProblemSolutions{\{f>\tau_k\},\Lambda}$, for all $k\in\N$. Applying \Cref{prop:stabilityGP}, we get $\{u\geq t\}\in\GeometricProblemSolutions{\{f\geq t\},\Lambda}$ for $t>0$.
	
	If $t<0$, we can find an increasing sequence $\sequence{\tau_k}{k}\subseteq (-\infty,t)$ such that $\{u< t\}=\union[k\in\N]\{u\leq\tau_k\}$, $\{f< t\}=\union[k\in\N]\{f\leq\tau_k\}$, $\{u\leq \tau_k\}\in\GeometricProblemSolutions{\{f\leq\tau_k\},\Lambda}$ for all $k\in\N$. Applying \Cref{prop:stabilityGP}, we get $\{u< t\}\in\GeometricProblemSolutions{\{f< t\},\Lambda}$ for $t<0$.
	
	By \Cref{rem:Solutions with complement datum}, we get to the thesis.
\end{proof}	

\begin{cor}\label{cor:regularity}
	Let $s\in(0,1)$, $\Lambda>0$ and $f\in\Lspace{1}(\Renne)$.
	If $u\in\FunctionalProblemSolutions{f,\Lambda}$, then all of its superlevel and sublevel sets 
	have boundary of class $C^1$ outside a closed singular set $S$ of Hausdorff
	dimension at most $n-3$.
\end{cor}
\begin{proof}
	The result follows combining \Cref{prop:equivalenceFuncGeomPBs}, \Cref{rem:zero super level is a local minimum} and \cite[Theorem 3.3]{CesNov17}.
\end{proof}

\section{High and low fidelity}\label{sec:lemmetto}

In the first part of this section we focus on the problem \ref{Pb:FunctionalProblem} when $f$ is binary (i.e. it is the characteristic function of some measurable set).

We begin with an immediate consequence of \Cref{prop:equivalenceFuncGeomPBs} and \Cref{rem:zero super level is a local minimum}.

\begin{prop}\label{cor:BINequivalenceFuncGeomPBs}
	Let $s\in(0,1)$, $\Lambda>0$ and $E\in\measurablesets$ such that $|E|<+\infty$.
	\begin{enumerate}[label=(\roman*)]
		\item\label{item:BINfromGPtoP} If $U\in\GeometricProblemSolutions{E,\Lambda}$ , then $\characteristic{U}\in\FunctionalProblemSolutions{\characteristic{E},\Lambda}$;
		\item\label{item:BINfromPtoGP}  		
		If	$u\in\FunctionalProblemSolutions{\characteristic{E},\Lambda}$, then $0\leq u\leq 1$ $\lebn$-a.e., $\{u>t\}\in\GeometricProblemSolutions{E,\Lambda}$ for all $t\in[0,1)$, $\{u\geq t\}\in\GeometricProblemSolutions{E,\Lambda}$ for all $t\in(0,1]$.
	\end{enumerate}
\end{prop}

\begin{remark}[Binary Data]\label{rem:EquivalenceProblems}
	\Cref{cor:BINequivalenceFuncGeomPBs} entails that a binary datum for \ref{Pb:FunctionalProblem} always implies the existence of binary solutions, although it is not guaranteed that every solution is of this form (see \Cref{thm:convexsCheeger}). In particular, if $u$ is an arbitrary solution of $\FunctionalProblem{\characteristic{E}}{\Lambda}$, then $\characteristic{\{u>t\}}\in\FunctionalProblemSolutions{\characteristic{E},\Lambda}$ for all $t\in[0,1)$ and $\characteristic{\{u\geq t\}}\in\FunctionalProblemSolutions{\characteristic{E},\Lambda}$ for all $t\in(0,1]$. See also \cite[Theorem 5.2]{ChaEse} and \cite[Proposition 2.2 and 2.3]{YinGolOsh} for the analogous results in the classical case.
	Furthermore, if $U$ is any solution of \ref{Pb:GeometricProblem}, then $\characteristic{U}\in\FunctionalProblemSolutions{\characteristic{E},\Lambda}$.
\end{remark}	

The following lemma is a straightforward adaptation of \cite[Corollary 4.1]{ChaEse}.
\begin{lemma}\label{lem:DiscontinuityParameter}
	Let $s\in(0,1)$, $f\in\Lspace{1}(\Renne)$ and for all $\Lambda>0$ let	
	\begin{equation*}
		\distSolDatPLUS(\Lambda)=\distSolDatPLUS(f,\Lambda):=\sup \Set{\norm{u-f}_{\Lspace{1}(\Renne)}}{u\in\FunctionalProblemSolutions{f,\Lambda}};
	\end{equation*}
	\begin{equation*}
		\distSolDatMINUS(\Lambda)=\distSolDatMINUS(f,\Lambda):=\inf \Set{\norm{u-f}_{\Lspace{1}(\Renne)}}{u\in\FunctionalProblemSolutions{f,\Lambda}};
	\end{equation*}
	\begin{equation*}
		\FunctionalProblemJumps{f}:=\Set{\Lambda>0}{\distSolDatMINUS(f, \Lambda)<\distSolDatPLUS(f, \Lambda)}.
	\end{equation*}
	Then, 
	\begin{equation}\label{eq:DistanceDecreasing}
		\distSolDatMINUS(\Lambda_1) \leq \distSolDatPLUS(\Lambda_1) \leq \distSolDatMINUS(\Lambda_2) \leq \distSolDatPLUS(\Lambda_2)
	\end{equation}
	for all $\Lambda_1\geq\Lambda_2>0$.
	
	Therefore, $\distSolDatPLUSMINUS$ are decreasing functions which are discontinuous at each point of $\FunctionalProblemJumps{f}$. 
\end{lemma}
\begin{proof}
	\newcommand{\greatFidelity}{\Lambda_{1}}
	\newcommand{\lowFidelity}{\Lambda_{2}}
	\newcommand{\ugreat}{u_{\greatFidelity}}
	\newcommand{\ulow}{u_{\lowFidelity}}
	We claim that for all $\greatFidelity>\lowFidelity>0$, if $\ugreat\in\FunctionalProblemSolutions{f, \greatFidelity}$ and $\ulow\in\FunctionalProblemSolutions{f, \lowFidelity}$, then
	\begin{equation*}
		\norm{\ulow-f}_{\Lspace{1}(\Renne)}\geq\norm{\ugreat-f}_{\Lspace{1}(\Renne)}.
	\end{equation*}
	We assume by contradiction that $\norm{\ulow-f}_{\Lspace{1}(\Renne)}<\norm{\ugreat-f}_{\Lspace{1}(\Renne)}$. By minimality of $\ulow$ it follows $\sEnergyLOneFracTV{(\ulow; f, \lowFidelity)}\leq\sEnergyLOneFracTV{(\ugreat; f, \lowFidelity)}$, which allows us to write:
	\begin{align*}
		&\sEnergyLOneFracTV{(\ulow; f, \greatFidelity)}=\sEnergyLOneFracTV{(\ulow; f, \lowFidelity)}+(\greatFidelity-\lowFidelity)\norm{\ulow-f}_{\Lspace{1}(\Renne)}\\
		&\leq \sEnergyLOneFracTV{(\ugreat; f, \lowFidelity)}+(\greatFidelity-\lowFidelity)\norm{\ulow-f}_{\Lspace{1}(\Renne)}\\
		&<\sEnergyLOneFracTV{(\ugreat; f, \lowFidelity)}+(\greatFidelity-\lowFidelity)\norm{\ugreat-f}_{\Lspace{1}(\Renne)}=\sEnergyLOneFracTV{(\ugreat; f, \greatFidelity)}.
	\end{align*}
Since $\ugreat$ is a solution of $\FunctionalProblem{f}{\greatFidelity}$, we get to a contradiction, thus proving the claim. \eqref{eq:DistanceDecreasing} simply follows from this claim and the definitions of $\distSolDatPLUSMINUS$. Finally, we fix $\Lambda\in\FunctionalProblemJumps{f}$ and, by \eqref{eq:DistanceDecreasing}, we observe that 
\begin{align*}
	\lim\limits_{\lambda\to\Lambda^-}\distSolDatPLUS(\lambda)=\lim\limits_{\lambda\to\Lambda^-}\distSolDatMINUS(\lambda)\geq\distSolDatPLUS(\Lambda)>\distSolDatMINUS(\Lambda)\geq\lim\limits_{\lambda\to\Lambda^+}\distSolDatPLUS(\lambda)=\lim\limits_{\lambda\to\Lambda^+}\distSolDatMINUS(\lambda).
\end{align*}
\end{proof}
Arguing exactly as in \cite[Corollary 5.3 and 5.4]{ChaEse}, we are able to prove that if the datum is the characteristic of a bounded and convex set, this suffices to establish uniqueness for $\FunctionalProblem{\characteristic{E}}{\Lambda}$ for almost every $\Lambda>0$. This is the content of the next statement.  

\newcommand{\Ulambda}{U(\Lambda)}
\begin{thm}\label{cor:UniquenessWithConvexDatum}
	Let $s\in(0,1)$ and $E\subset\Renne$ be bounded and convex. Then, for $\lebone-$a.e. $\Lambda>0$ (precisely, for all $\Lambda\in(0,+\infty)\setminus \FunctionalProblemJumps{\characteristic{E}}$) the problem $\FunctionalProblem{\characteristic{E}}{\Lambda}$ admits a unique solution and
	\begin{equation*}
		\FunctionalProblemSolutions{\characteristic{E},\Lambda}=\{\characteristic{\Ulambda}\},
	\end{equation*}
	for some $\Ulambda\in\measurablesets$ such that $\Ulambda\subseteq E$.
\end{thm}

\begin{proof}[Proof of \Cref{cor:UniquenessWithConvexDatum}]
	\newcommand{\ulambda}{u_{\Lambda}}
	\newcommand{\ulambdaone}{u_{\Lambda}^1}
	\newcommand{\ulambdatwo}{u_{\Lambda}^2}
	\newcommand{\tone}{t_{1}}
	\newcommand{\ttwo}{t_{2}}
	\newcommand{\ti}{t_{i}}
	We first observe that, thanks to \Cref{lem:DiscontinuityParameter}, $\FunctionalProblemJumps{\characteristic{E}}$ is $\lebone$-negligible and we fix $\Lambda\in (0,+\infty)\setminus \FunctionalProblemJumps{\characteristic{E}}$.
	
	As first step, we aim to prove that every solution of $\FunctionalProblem{\characteristic{E}}{\Lambda}$ is binary (cf. \Cref{sec:notation}) with support contained in $E$.
	Let $\ulambda\in\FunctionalProblemSolutions{\characteristic{E},\Lambda}$.
	So, $0\leq \ulambda\leq 1$ $\lebn$-a.e., $\{\ulambda>t\}\in\GeometricProblemSolutions{E,\Lambda}$ and $\characteristic{\{\ulambda>t\}}\in\FunctionalProblemSolutions{\characteristic{E},\Lambda}$ for all $t\in[0,1)$ by \Cref{cor:BINequivalenceFuncGeomPBs}.
	Let $1>\tone>\ttwo\geq 0$ and assume that $\{\ulambda>\tone\}$ is strictly contained in $\{\ulambda>\ttwo\}$ (upon negligible sets). For every $i\in\{1,2\}$, by convexity of $E$, we get
	\begin{align*}
		\Ps(E\cap \{\ulambda>\ti\})\leq\Ps(\{\ulambda>\ti\}).
	\end{align*}
	So, by minimality of $\{\ulambda>\ti\}$, we obtain $\{\ulambda>\tone\}\subsetneqq \{\ulambda>\ttwo\} \subseteq E$ (note that $|(E\cap \{\ulambda>\ti\})\Delta E|$ would be strictly less than $|\{\ulambda>\ti\}\Delta E|$ if $\{\ulambda>\ti\}$ was not contained in $E$). Therefore, $\characteristic{\{\ulambda>\tone\}},\characteristic{\{\ulambda>\ttwo\}}\in\FunctionalProblemSolutions{\characteristic{E},\Lambda}$ and $\norm{\characteristic{\{\ulambda>\tone\}}-\characteristic{E}}_{\Lspace{1}(\Renne)}>\norm{\characteristic{\{\ulambda>\ttwo\}}-\characteristic{E}}_{\Lspace{1}(\Renne)}$. This means that $\distSolDatMINUS(\Lambda,\characteristic{E})<\distSolDatPLUS(\Lambda,\characteristic{E})$ and, consequently, that $\Lambda\in \FunctionalProblemJumps{\characteristic{E}}$, which is a contradiction. Therefore, there exists $\Ulambda\in\measurablesets$ such that $\Ulambda\subseteq E$ and $\{\ulambda>t\}=\Ulambda$ for each $t\in[0,1)$. Being $0\leq \ulambda\leq 1$ $\lebn$-a.e., we conclude that $\ulambda=\characteristic{\Ulambda}$ $\lebn$-a.e. 	
	
	Finally, we prove that each (binary) solution of $\FunctionalProblem{\characteristic{E}}{\Lambda}$ is uniquely determined. This is true since, if $\ulambdaone,\ulambdatwo\in\FunctionalProblemSolutions{\characteristic{E},\Lambda}$ then also $\unmezzo\ulambdaone+\unmezzo\ulambdatwo$is a solution of $\FunctionalProblem{\characteristic{E}}{\Lambda}$, by convexity of the energy $\sEnergyLOneFracTV{(\cdot,\characteristic{E},\Lambda)}$. Since all solutions are binary, this implies that $\ulambdaone=\ulambdatwo$ $\lebn$-a.e. and the previous step allows us to conclude that they are also equal $\lebn$-a.e. to $\characteristic{\Ulambda}$ for some $\Ulambda\in\measurablesets$ such that $\Ulambda\subseteq E$.
\end{proof}

Now we are interested in studying the behaviour of solutions of \ref{Pb:FunctionalProblem} for large values of the fidelity parameter, under some regularity assumptions on the datum. The key result is the following:

\begin{thm}
	\label{thm:lemmetto per lavoro su variazione seconda}
	Let $s\in(0,1)$ and $E\subseteq\Renne$ with $\partial E\in C^{1,1}$ either bounded or with bounded complement.
	Then, there exists $\Lambda(E,s)>0$ s.t. for every $\Lambda\geq\Lambda(E,s)$, the set $E$ is the unique solution of the problem \ref{Pb:GeometricProblem}. That is,
	\begin{equation*}
	\Ps(E)-\Ps(\U)\leq\Lambda|E\Delta \U|
	\end{equation*}
	for every $\U\in\measurablesets$ and $\Lambda\geq\Lambda(E,s)$.
\end{thm}
\begin{remark}\label{rem:lemmetto}
	From the proof of \Cref{thm:lemmetto per lavoro su variazione seconda}, we deduce that we can choose $\Lambda(E,s)=\frac{\cn}{\omegan}\frac{1}{{\rzero}^s}$, where $\rzero=\rzero(E)>0$ is such that both $E$ and its complement are the union of balls of radius $\rzero$.
\end{remark}

Before moving to the proof of \Cref{thm:lemmetto per lavoro su variazione seconda} we need some auxiliary results.
\begin{lemma}[maximal and minimal solutions]
	\label{lem:maximum and minimum solutions}
	Let $s\in(0,1)$, $\Lambda>0$ and $E\in\measurablesets$ be bounded.
	Then \ref{Pb:GeometricProblem} admits a minimal and a maximal (w.r.t. set inclusion) solution and they are bounded.	
\end{lemma}
\begin{remark}
	If \ref{Pb:GeometricProblem} admits a minimal (respectively maximal) solution, it is unique upon $\lebn$-negligible sets.  
\end{remark}
\begin{proof}
	\newcommand{\A}{\GeometricProblemSolutions{E,\Lambda}}
	We follow an argument similar to the one in \cite[Proposition 6.1]{CMP}.
	
	Let $R>0$ s.t. $E\subset\BR$ and $\U\in\measurablesets$. Without loss of generality, we assume $|U|<+\infty$, otherwise $\sEnergyGeometric{(U; E, \Lambda)=+\infty}$. Then, $\Ps(\U\cap\BR)\leq\Ps(\U)$ since $\BR$ is convex, and $|E\Delta(\U\cap\BR)|=|(E\cap\BR)\Delta(\U\cap\BR)|\leq|E\Delta \U|$. Furthermore, the last inequality is strict when $\U\nsubseteqq\BR$. So \ref{Pb:GeometricProblem} is indeed equivalent to the following minimum problem:
	\begin{align}
	\label{eq:EquivPB bound}
	\min_{\U\in\measurablesets, \U\subseteq\BR}\Ps(\U)+\Lambda|E\Delta \U|.
	\end{align}
	We now make the following assertions:
	\begin{claim}{1}
		\label{claim 1}
		$\A\neq\emptyset$.
	\end{claim}
		This is a consequence of \Cref{prop:existenceL1Datum} and \Cref{cor:BINequivalenceFuncGeomPBs} part $\ref{item:BINfromPtoGP}$.		
	\begin{claim}{2}
		$\A$ is stable under finite intersection and finite union.
	\end{claim}
		Let $\Uone$ and $\Utwo$ be two solutions of \ref{Pb:GeometricProblem}. 
		By \Cref{cor:BINequivalenceFuncGeomPBs} part $\ref{item:BINfromGPtoP}$,  $\characteristic{\Uone}$ and $\characteristic{\Utwo}$ solve $\FunctionalProblem{\characteristic{E}}{\Lambda}$. By \Cref{rem:ConvexityFunctionalProblemSolutions}, $\unmezzo\left(\characteristic{\Uone}+\characteristic{\Utwo}\right)$ belongs to $\FunctionalProblemSolutions{\characteristic{E}, \Lambda}$. Applying \Cref{cor:BINequivalenceFuncGeomPBs} part $\ref{item:BINfromPtoGP}$ we deduce that for every $t\in[0,1)$ the set $\{\unmezzo\left(\characteristic{\Uone}+\characteristic{\Utwo}\right)>t\}$ belongs to $\A$, which involves $\Uone\cap\Utwo\in\A$ and $\Uone\cup\Utwo\in\A$.
	\begin{claim}{3}
		$\A$ is stable under countable decreasing intersections and under countable increasing unions.
	\end{claim}
		This follows from \Cref{prop:stabilityGP}, since in both cases the limits are $\Lspace{1}(\Renne)$-limits.

	\begin{claim}{4}
		$\A$ admits a minimum and a maximum (w.r.t. set inclusion).
	\end{claim}
		As far as the existence of a minimum solution is concerned, the strategy is to show that there exists $\Emini\in\A$ minimizing the Lebesgue measure and then to prove by contradiction that it is contained in all sets of $\A$. The argument for $\Emaxi$ is left to the reader, being very similar to the previous one. We put		
		\begin{align*}
		m:=\inf\bigSet{|\U|}{\U\in\A}
		\end{align*}		
		and we let $\sequence{\Un}{n}\subseteq\A$ such that $|\Un|\to m\in[0,+\infty)$ as $n$ goes to $+\infty$. For every $k\in\N$, we put $\Vk:=\bigcap_{n=1}^{k}\Un$ and note that $\sequence{\Vk}{k}\subseteq\A$ is a decreasing sequence, converging in $\Lspace{1}(\Renne)$ to $\Emini:=\bigcap_{k=1}^{+\infty}\Vk\in\A$. In addition, we observe that $|\Emini|=m$.
		Now, we suppose by contradiction there exists $\Etilde\in\A$ such that $\Emini\nsubseteq \Etilde$. Then $\Emini\cap\Etilde\in\A$, so that:
		\begin{align*}
		|\Emini\cap\Etilde|=|\Emini|-|\Emini\setminus\Etilde|=m-|\Emini\setminus\Etilde|<m\leq|\Emini\cap\Etilde|,
		\end{align*}
		which is our contradiction.
	
	The proof is now concluded.
\end{proof}

The following statement is a direct consequence of \Cref{lem:maximum and minimum solutions} and \Cref{rem:Solutions with complement datum}.
 
\begin{cor}[maximal and minimal solutions 2]
	\label{cor:maximum and minimum solutions 2}
	Let $s\in(0,1)$, $\Lambda>0$ and $E\in\measurablesets$ be bounded. Then $\GeometricProblem{\comp{E}}{\Lambda}$ admits a minimal and a maximal solution.
	Furthermore, \eqref{eq:solutions with complement datum} holds.
	
\end{cor}

\begin{defn}\label{defn:MaxMinSolutions}
	For fixed $s\in(0,1)$, $\Lambda>0$ and $E\in\measurablesets$ (either bounded or with bounded complement)
	we denote by $\Emini$ and $\Emaxi$ respectively the \textit{minimal} and the \textit{maximal} solution (w.r.t set inclusion) to \ref{Pb:GeometricProblem}, which exist thanks to \Cref{lem:maximum and minimum solutions} and \Cref{cor:maximum and minimum solutions 2}. In particular, \Cref{cor:maximum and minimum solutions 2} ensures that $\mini{(\comp{E})}=\comp{(\Emaxi)}$ and $\maxi{(\comp{E})}=\comp{(\Emini)}$.
\end{defn}

\begin{lemma}[Comparison]
	\label{lem:Comparison}
	Let $s\in(0,1)$, $\Lambda>0$ and $\Eone,\Etwo\in\measurablesets$ such that for all $i\in\{1,2\}$  $\Ps(\Ei)<\infty$ and either $\Ei$ or $\comp{\Ei}$ is bounded .
	
	If $\Etwo\subseteq\Eone$, then $\Etwomini\subseteq\Eonemini$ and $\Etwomaxi\subseteq\Eonemaxi$.
\end{lemma}

\begin{remark}
	Compare \Cref{lem:Comparison} with the question about the existence of a comparison principle that the authors ask themselves in \cite[pages 1826-1827]{ChaEse} and the answer given in \cite[Theorem 3.1]{YinGolOsh}.
\end{remark}

\begin{proof}[Proof of \Cref{lem:Comparison}]
	\newcommand{\sopraFunzione}{f}
	\newcommand{\sottoFunzione}{g}
	$\forallionetwo$ let $\ULambdai$ be a solution of $\GeometricProblem{\Ei}{\Lambda}$.
	
	Thanks to the minimality of $\ULambdaone$ and $\ULambdatwo$ we find
	\begin{equation}
	\label{eq:beta}
	\begin{split}
	&\Ps(\ULambdaone)+|\ULambdaone\Delta\Eone|+\Ps(\ULambdatwo)+|\ULambdatwo\Delta\Etwo|\\
	&\leq\Ps(\ULambdaone\cup\ULambdatwo)+|(\ULambdaone\cup\ULambdatwo)\Delta\Eone|\\
	&+\Ps(\ULambdaone\cap\ULambdatwo)+|(\ULambdaone\cap\ULambdatwo)\Delta\Etwo|.
	\end{split}
	\end{equation}
	
	By the submodularity property of the fractional perimeter we get
	\begin{align}
	\label{eq:alpha 1}
	\Ps(\ULambdaone)+\Ps(\ULambdatwo)\geq\Ps(\ULambdaone\cap\ULambdatwo)+\Ps(\ULambdaone\cup\ULambdatwo);
	\end{align}
	
	while an elementary computation leads to	
	\begin{align}
	\label{eq:alpha 2}
	|\ULambdaone\Delta\Eone|+|\ULambdatwo\Delta\Etwo|\geq|(\ULambdaone\cup\ULambdatwo)\Delta\Eone|+|(\ULambdaone\cap\ULambdatwo)\Delta\Etwo|.
	\end{align}
	 Precisely, to prove \eqref{eq:alpha 2} we define the following auxiliary functions
	\begin{align*}
		&\sopraFunzione:=\left|\characteristic{\ULambdaone}-\characteristic{\Eone}\right|+\left|\characteristic{\ULambdatwo}-\characteristic{\Etwo}\right|,\\
		&\sottoFunzione:=\left|\characteristic{\ULambdaone\cup\ULambdatwo}-\characteristic{\Eone}\right|+\left|\characteristic{\ULambdaone\cap\ULambdatwo}-\characteristic{\Etwo}\right|.
	\end{align*}
	Then, we observe that $\sopraFunzione$ and $\sottoFunzione$ coincide $\lebn$-a.e. in $\ULambdaone\setminus\ULambdatwo$, in $\ULambdaone\cap\ULambdatwo$ and in $\comp{\left(\ULambdaone\cup\ULambdatwo\right)}$. Furthermore, since the inclusion $\Etwo\subseteq\Eone$ holds, $\sopraFunzione\geq\sottoFunzione$ $\lebn$-a.e. in $\ULambdatwo\setminus\ULambdaone$, which allows us to deduce that $\sopraFunzione\geq\sottoFunzione$ $\lebn$-a.e. in $\Renne$. By integrating over $\Renne$ both sides of this pointwise inequality, we finally get \eqref{eq:alpha 2}. 
	Combining \eqref{eq:alpha 1} and \eqref{eq:alpha 2} we deduce that the inequality \eqref{eq:beta} is in fact an equality. This new information joint with \eqref{eq:alpha 1} and \eqref{eq:alpha 2} again leads to:
	
	\begin{align}
	\label{eq:alpha 1 is equality}
	\Ps(\ULambdaone)+\Ps(\ULambdatwo)=\Ps(\ULambdaone\cap\ULambdatwo)+\Ps(\ULambdaone\cup\ULambdatwo);
	\end{align}
	\begin{align}
	\label{eq:alpha 2 is equality}
	|\ULambdaone\Delta\Eone|+|\ULambdatwo\Delta\Etwo|=|(\ULambdaone\cup\ULambdatwo)\Delta\Eone|+|(\ULambdaone\cap\ULambdatwo)\Delta\Etwo|.
	\end{align}
	\newcommand{\A}{\ULambdaone}
	\newcommand{\B}{\ULambdatwo}
\newcommand{\interazione}[2]{\interaction\big(#1\times#2\big)}
	\eqref{eq:alpha 1 is equality} can be rewritten, according to \eqref{eq:interaction measure}, as follows:
	\begin{align*}
	&\interaction\big(\left((\A\setminus \B)\sqcup(\A\cap \B)\right)\times\left(\comp{(\A\cup \B)}\sqcup(\B\setminus \A)\right)\big)\\
	&+\interaction\big(\left((\B\setminus \A)\sqcup(\A\cap \B)\right)\times\left(\comp{(\A\cup \B)}\sqcup(\A\setminus \B)\right)\big)\\
	&=\interaction\big((\A\cap \B)\times\left((\A\setminus \B)\sqcup(\B\setminus \A)\sqcup\comp{(\A\cup \B)}\right)\big)\\
	&+\interaction\big(\left((\A\setminus \B)\sqcup(\A\cap \B)\sqcup(\B\setminus \A)\right)\times\comp{(\A\cup \B)}\big).
	\end{align*}
    By the additivity of the measure $\interaction$, the last equation becomes
	\begin{align*}
		&\interazione{(\A\setminus\B)}{\comp{(\A\cup\B)}}
		+\interazione{(\A\setminus\B)}{(\B\setminus\A)}\\
		&+\interazione{(\A\cap\B)}{\comp{(\A\cup\B)}}
		+\interazione{(\A\cap\B)}{(\B\setminus\A)}\\
		&+\interazione{(\B\setminus\A)}{\comp{(\A\cup\B)}}
		+\interazione{(\B\setminus\A)}{(\A\setminus\B)}\\
		&+\interazione{(\A\cap\B)}{\comp{(\A\cup\B)}}
		+\interazione{(\A\cap\B)}{(\A\setminus\B)}\\
		&=\interazione{(\A\cap\B)}{(\A\setminus\B)}
		+\interazione{(\A\cap\B)}{(\B\setminus\A)}\\
		&+\interazione{(\A\cap\B)}{\comp{(\A\cup\B)}}
		+\interazione{(\A\setminus\B)}{\comp{(\A\cup\B)}}\\
		&+\interazione{(\A\cap\B)}{\comp{(\A\cup\B)}}
		+\interazione{(\B\setminus\A)}{\comp{(\A\cup\B)}}.
	\end{align*}
	
	Then, we easily obtain $\interaction((\ULambdaone\setminus\ULambdatwo)\times (\ULambdatwo\setminus\ULambdaone))=0$, from which we conclude that
	\begin{equation}\label{eq:disjunction}
	\ULambdaone\subseteq\ULambdatwo \text{ or } \ULambdatwo\subseteq\ULambdaone.
	\end{equation}
	
	Let us assume $\ULambdaone\subseteq\ULambdatwo$.
	
	Then, for every $i\in\onetwo$ and $j\in\onetwo\setminus \{i\}$ it can be shown that $\ULambdai\in\GeometricProblemSolutions{\Ej, \Lambda}$. Precisely, by minimality of $\ULambdai$ for $\GeometricProblem{\Ei}{\Lambda}$, we get
	\begin{align}
	\label{eq:key}
	\Ps(\ULambdai)+\Lambda|\ULambdai\Delta\Ei|\leq\Ps(\ULambdaj)+\Lambda|\ULambdaj\Delta\Ei|.
	\end{align}
	Now we rewrite \eqref{eq:alpha 2 is equality} using the information $\ULambdaone\subseteq\ULambdatwo$:
	\begin{align}
	\label{eq:alpha 2 is equality revisited}
	|\ULambdai\Delta\Ei|+|\ULambdaj\Delta\Ej|=|\ULambdaj\Delta\Ei|+|\ULambdai\Delta\Ej|.
	\end{align}    
	Combining $\eqref{eq:key}$ and $\eqref{eq:alpha 2 is equality revisited}$ we deduce $\ULambdai$ solves $\GeometricProblem{\Ej}{\Lambda}$ for every $i,j\in\onetwo$ s.t. $i\neq j$.
		
	Then, we showed:
	\begin{equation}\label{eq:implication}
	\ULambdaone\subseteq\ULambdatwo\Rightarrow \begin{dcases*}
	\ULambdatwo\subseteq\Eonemaxi\\
	\Etwomini\subseteq\ULambdaone
	\end{dcases*}.
	\end{equation}
		
	This allows us to conclude. Precisely, we take first $\ULambdaone=\Eonemini$ and $\ULambdatwo=\Etwomini$. Then, by \eqref{eq:disjunction}, $\Etwomini\subseteq\Eonemini$ or $\Eonemini\subseteq\Etwomini$. In the former case we are done, while the latter still implies $\Etwomini\subseteq\Eonemini$, thanks to \eqref{eq:implication}. Finally, we take $\ULambdaone=\Eonemaxi$ and $\ULambdatwo=\Etwomaxi$ and we conclude by a similar argument.
\end{proof}

\begin{lemma}[Uniqueness for $\GeometricProblemSolutions{\Brx, \Lambda}$ with $\Lambda$ large]
	\label{lem:Uniqueness for Balls}
	Let $s\in(0,1)$. For all $x\in\Renne, r>0$ and $\Lambda>\frac{\cn}{\omegan}\frac{1}{r^s}$
	the problem $\GeometricProblem{\BallRadiusCenter{r}{x}}{\Lambda}$ admits $\BallRadiusCenter{r}{x}$ as unique solution.
\end{lemma}
\begin{proof}
	Let $r>0$, $x\in\Renne$ and $\Lambda>\frac{\cn}{\omegan}\frac{1}{r^s}$.
	Note that the argument used to obtain \eqref{eq:EquivPB bound} in the proof of \Cref{lem:maximum and minimum solutions}, tells us that every solution of $\GeometricProblem{\BallRadiusCenter{r}{x}}{\Lambda}$ is contained in $\BallRadiusCenter{r}{x}$. 
	
	Now, we consider $\U\in\measurablesets$ such that $\U\subseteq\BallRadiusCenter{r}{x}$. Then,
	\begin{align}
	\label{eq:fractional isoperimetric}
	\Ps(\U)\geq\Ps(\BU),
	\end{align}
	and the inequality in \eqref{eq:fractional isoperimetric} is an equality if and only if $\U$ is a ball. For a proof of \eqref{eq:fractional isoperimetric} see for instance \cite[Theorem A.1, case $p=1$]{FraSei} or \cite[Proposition 3.1]{CesNov}.
	
	Furthermore, for every $\BUy$ contained in $\Brx$:	
	\begin{align}\label{eq:ball and symmetric difference}
	|\Brx\Delta \U|=|\Brx|-|\U|=|\Brx|-|\BUy|=|\BUy\Delta\Brx|.
	\end{align}
	
	Combining \eqref{eq:fractional isoperimetric} and \eqref{eq:ball and symmetric difference}, we obtain that every solution of $\GeometricProblem{\BallRadiusCenter{r}{x}}{\Lambda}$ must be a ball contained in $\BallRadiusCenter{r}{x}$. That is,
	\begin{align}
	\label{eq:key2 eq for ball pb}
	\inf_{\substack{B\subset\Brx\\ B \text{ ball}}}\PsiLambda{B} \text{\hspace*{0.1 cm} is equivalent to }\GeometricProblem{\Brx}{\Lambda}.
	\end{align}
	
	At this point, thanks to the scaling property of $\Ps$, for every $\rho>0$ we obtain
	\begin{align}
	\label{eq:problem one parameter}
	\PsiLambda{\Brhoy}=\cn\rho^{n-s}+\Lambda\omegan(r^n-\rho^n),
	\end{align}
	whenever $\Brhoy\subset\Brx$.
	Thanks to our choice of $\Lambda$, it can easily be seen that the function in \eqref{eq:problem one parameter} admits a unique global minimum, attained when $\rho$ is equal to $r$.
	From \eqref{eq:key2 eq for ball pb} we deduce that $\Brx$ is the unique solution to $\GeometricProblem{\Brx}{\Lambda}$.
\end{proof}

\begin{cor}
	\label{cor: Uniqueness for Balls (uniform version)}
	Let $s\in(0,1)$ and $\rzero>0$.
	For all $x\in\Renne, r\geq\rzero$ and $\Lambda>\frac{\cn}{\omegan}\frac{1}{{\rzero}^s}$
	the problem $\GeometricProblem{\BallRadiusCenter{r}{x}}{\Lambda}$ admits $\BallRadiusCenter{r}{x}$ as unique solution.
\end{cor}
\begin{proof}
	It is a direct consequence of \Cref{lem:Uniqueness for Balls}.	
\end{proof}

\begin{cor}
	\label{cor: Uniqueness for CO-Balls (uniform version)}
	Let $s\in(0,1)$ and $\rzero>0$.
	For all $x\in\Renne, r\geq\rzero$ and $\Lambda>\frac{\cn}{\omegan}\frac{1}{{\rzero}^s}$
	the problem $\GeometricProblem{\comp{\BallRadiusCenter{r}{x}}}{\Lambda}$ admits $\comp{\BallRadiusCenter{r}{x}}$ as unique solution.
\end{cor}
\begin{proof}
	It follows from \Cref{cor: Uniqueness for Balls (uniform version)} and \Cref{cor:maximum and minimum solutions 2}.
\end{proof}
Now we are ready to prove \Cref{thm:lemmetto per lavoro su variazione seconda}.

\begin{proof}[Proof of \Cref{thm:lemmetto per lavoro su variazione seconda}]
	\newcommand{\Imeno}{I^{-}}
	\newcommand{\Ipiu}{I^{+}}
	\newcommand{\xmeno}{x^{-}}
	\newcommand{\xpiu}{x^{+}}
	\newcommand{\Brxmeno}{\BallRadiusCenter{r(x^{-})}{x^{-}}}
	\newcommand{\Brxpiu}{\BallRadiusCenter{r(x^{+})}{x^{+}}}
	Thanks to the boundary regularity of $E$, we can find a positive constant $\rzero>0$, two countable sets of indexes $\Imeno\subset E$ and $\Ipiu\subset\comp{E}$ and a positive function $r$ defined on $\Ipiu\cup\Imeno$ such that $r\geq\rzero\on\Ipiu\cup\Imeno$ and:
	\begin{align*}
	&\Brxmeno\subseteq E\spazio \forall\xmeno\in\Imeno, &\union[\xmeno\in\Imeno]\Brxmeno=E,\\
	&E\subseteq\comp{\Brxpiu}	\spazio \forall\xpiu\in\Ipiu, &\intersection[\xpiu\in\Ipiu]\comp{\Brxpiu}=\closure{E}.
	\end{align*}
	
	Let $\Lambda>\frac{\cn}{\omegan}\frac{1}{{\rzero}^s}$, $\xmeno\in\Imeno$ and $\xpiu\in\Ipiu$.
	Let $\Emini$ and $\Emaxi$ be the minimum and the maximum solution to $\GeometricProblem{\Lambda}{E}$ (see \Cref{defn:MaxMinSolutions}).
	
	Since $\Brxmeno\subseteq E$, by \Cref{lem:Comparison} and \Cref{cor: Uniqueness for Balls (uniform version)} we get:
	\begin{align}
	\label{eq:interno}
	\Brxmeno\subseteq\Emini.
	\end{align}
	Furthermore, $E\subseteq\comp{\Brxpiu}$ together with \Cref{lem:Comparison} and \Cref{cor: Uniqueness for CO-Balls (uniform version)} implies:	
	\begin{align}
	\label{eq:esterno}
	\Emaxi\subseteq\comp{\Brxpiu}.
	\end{align}
	Taking the union as $\xmeno\in\Imeno$ in \eqref{eq:interno} and the intersection as $\xpiu\in\Ipiu$ in \eqref{eq:esterno} we obtain:
	\begin{align*}
	E\subseteq \Emini\subseteq\Emaxi\subseteq \closure{E}.
	\end{align*}
\end{proof}

The next result characterizes the solutions of \ref{Pb:FunctionalProblem} for high values of the fidelity parameter under regularity assumptions on the datum. It is the analogue of \cite[Theorem 5.6]{ChaEse} in the nonlocal setting.

\begin{thm}[High fidelity]\label{thm:highFidelity}
	Let $s\in(0,1)$ and $f\in\Lspace{1}(\Renne)$ with superlevel sets uniformly bounded in $C^{1,1}$ (i.e., there exists $r>0$ such that every superlevel set of $f$ and its complement are the union of balls of radius $r$).
	Then, there exists $\Lambda(f,s)>0$ s.t.	
	\begin{equation*}
		\FunctionalProblemSolutions{f,\Lambda}=\{f\},
	\end{equation*}
	for every $\Lambda\geq\Lambda(f,s)$.
\end{thm}

\begin{remark}
	We do not adapt the proof of \cite[Theorem 5.6]{ChaEse}, which is done via calibrations, but we prove \Cref{thm:highFidelity} applying \Cref{thm:lemmetto per lavoro su variazione seconda}.
\end{remark}
\begin{proof}
	We already know $\FunctionalProblemSolutions{f,\Lambda}\neq\emptyset$ for every $\Lambda>0$.
	
	By \Cref{thm:lemmetto per lavoro su variazione seconda}, \Cref{rem:lemmetto} and the uniform boundedness in $C^{1,1}$ of the superlevel sets of the datum $f$, there exists $\Lambda(f,s)>0$ such that $\GeometricProblemSolutions{\{f>t\},\Lambda}=\{\{f>t\}\}$ for every $\Lambda\geq\Lambda(f,s)$ and $t\in\R$.
	
	By \Cref{prop:equivalenceFuncGeomPBs} $\ref{item:fromPtoGP}$, we deduce that if $u\in\FunctionalProblemSolutions{f,\Lambda}$, then $\{u>t\}=\{f>t\}$ for every $t\in\R\setminus\{0\}$. Then, $u=f$ almost everywhere.
	
\end{proof}

\begin{thm}[Low fidelity]\label{thm:lowFidelity}
	Let $s\in(0,1)$ and $R>0$ and $f\in\Lspace{1}(\Renne)$ such that $\supp(f)\subset\BR$. Then, there exists $\Lambda(R,n,s)>0$ such that \begin{align*}
	\FunctionalProblemSolutions{f,\Lambda}=\{0\}
	\end{align*}
	for any $0<\Lambda<\Lambda(R,n,s)$.
\end{thm}
\begin{proof}
	We recall the following fractional Sobolev inequality \cite[see][Theorem 6.5]{DiNPalVal}: there exists a constant $C(n,s)>0$ such that
	\begin{align}\label{eq:Ws1embeddinginLpstar}
	\norm{u}_{\Lspace{\frac{n}{n-s}}(\Renne)}\leq C(n,s)\sTotalVar{u}
	\end{align}
	for all $u\in\Wspace{s}{1}(\Renne)$.
	
	Let $0<\Lambda<\Lambda(R,n,s):=\left(C(n,s)R^s\omegan^{\frac{s}{n}}\right)^{-1}$  and $u\in\FunctionalProblemSolutions{f,\Lambda}$. So, $\sEnergyLOneFracTV{(u;f,\Lambda)}\leq\sEnergyLOneFracTV{(0;f,\Lambda)}$ and an application of \eqref{eq:Ws1embeddinginLpstar} entails:
	\begin{align*}
	\frac{1}{C(n,s)}\norm{u}_{\Lspace{\frac{n}{n-s}}(\Renne)}+\Lambda\norm{u-f}_{\Lspace{1}(\Renne)}\leq\Lambda\norm{f}_{\Lspace{1}(\Renne)}=\Lambda\norm{f}_{\Lspace{1}(\BR)}.
	\end{align*}
	Then, by \Holder's inequality we get $\norm{u}_{\Lspace{1}(\BR)}\leq\norm{u}_{\Lspace{\frac{n}{n-s}}(\BR)}\norm{1}_{\Lspace{\frac{n}{s}}(\BR)}$. Therefore:
	\begin{align*}
	\Lambda(R,n,s)\norm{u}_{\Lspace{1}(\Br)}+\frac{1}{C(n,s)}\norm{u}_{\Lspace{\frac{n}{n-s}}(\comp{\BR})}+\Lambda\norm{u-f}_{\Lspace{1}(\BR)}\leq\Lambda\norm{f}_{\Lspace{1}(\BR)},
	\end{align*}
	which leads to
	\begin{align*}
	&\left(\Lambda(R,n,s)-\Lambda\right)\norm{u}_{\Lspace{1}(\BR)}+\frac{1}{C(n,s)}\norm{u}_{\Lspace{\frac{n}{n-s}}(\comp{\BR})}\\
	&\leq\Lambda(\norm{f}_{\Lspace{1}(\BR)}-\norm{u}_{\Lspace{1}(\BR)}-\norm{u-f}_{\Lspace{1}(\BR)})\leq 0;
	\end{align*}
	and so, $\norm{u}_{\Lspace{1}(\BR)}=\norm{u}_{\Lspace{\frac{n}{n-s}}(\comp{\BR})}=0$, allowing us to conclude.
	
\end{proof}

We now recall the following definition from \cite{BraLinPar}.

\begin{defn}[s-Cheeger set]\label{def:sCheeger}
	Let $s\in(0,1)$ and $E\in\measurablesets$ be a bounded set with nonempty interior. Then,
	\begin{equation}\label{eq:sCheegerConstant}
		\Cheegs(E):=\inf_{U\in\measurablesets, U\subseteq E}\frac{\Ps(U)}{|U|}
	\end{equation}
is the \textit{s-Cheeger constant} of $E$. A set $U=U(E)$ achieving the infimum in \eqref{eq:sCheegerConstant} is said to be an \textit{s-Cheeger set} of $E$, and we let
\begin{equation*}
	\sCheegerClass(E):=\Set{U}{U \text{ s-Cheeger set of }E}\cup \{\emptyset\}.
\end{equation*}
Furthermore, $E$ is said to be \textit{s-calibrable} if it is an s-Cheeger set of itself; namely,
\begin{equation}\label{eq:sCalibrable}
	\Cheegs(E)=\frac{\Ps(E)}{|E|}.
\end{equation}
\end{defn}
\begin{remark}
	 Balls are examples of s-calibrable sets of $\Renne$, as shown in \cite[Remark 5.2]{BraLinPar}.
	 Moreover, the minimum problem \eqref{eq:sCheegerConstant} always admits solutions if $E$ is bounded with non-empty interior \cite[cf.][Proposition 5.3]{BraLinPar}. Nevertheless, to the best of our knowledge, it is not known whether uniqueness holds if $E$ is convex, although this is true for the classical Cheeger problem.	Finally, we highlight that \eqref{eq:sCheegerConstant} differs by a multiplicative factor $2$ from the corresponding definition in \cite{BraLinPar}. 
\end{remark}
The following result refines \Cref{cor:UniquenessWithConvexDatum}, providing more information on the solutions of $\FunctionalProblem{\characteristic{E}}{\Lambda}$ if the set $E$ is bounded and convex. 
\begin{thm}\label{thm:convexsCheeger}
	Let $s\in(0,1)$ and $E\subset\Renne$ be a bounded convex set with nonempty interior.
	Then, $\FunctionalProblem{\characteristic{E}}{\Lambda}$ admits a unique solution for $\lebone-$a.e. $\Lambda>0$. Furthermore,
	\begin{enumerate}[label=(\roman*)]
		\item\label{item:Sone} if $0<\Lambda<\Cheegs(E)$, then $\FunctionalProblemSolutions{\characteristic{E},\Lambda}=\{0\}$;
		\item\label{item:Stwo} if $\Lambda=\Cheegs(E)$, then
		$\GeometricProblemSolutions{E, \Lambda}=\sCheegerClass(E)$ and \\
		$\FunctionalProblemSolutions{\characteristic{E},\Lambda} = \Set{u\in\Wspace{s}{1}(\Renne;[0,1])}{\{u>t\} \in \sCheegerClass(E) \text{ for all }t\in [0,1)}$;
		\item\label{item:Sthree} if $\Lambda>\Cheegs(E)$, then $0\notin\FunctionalProblemSolutions{\characteristic{E},\Lambda}$.
	\end{enumerate}
\end{thm}
\begin{proof}
	We start by observing that, since $E$ is bounded and convex, \ref{Pb:GeometricProblem} is equivalent to:
	\begin{equation}\label{Pb:GeomProbConvexDatum}
		\min_{U\in\measurablesets, U\subseteq E}\Ps(U)-\Lambda|U|.
	\end{equation}
	\newcommand{\Lambdastardown}{\Lambda_{*}}
	We put
	\begin{equation*}
		\Lambdastardown:=\sup\Set{\Lambda>0}{0\in\FunctionalProblemSolutions{\characteristic{E}, \Lambda}}.
	\end{equation*}
Since $E$ is bounded, thanks to \Cref{thm:lowFidelity}, $0<\Lambdastardown\leq +\infty$. 

We now claim that $\Lambdastardown=\Cheegs(E)$.

First, we fix $0<\Lambda\leq\Cheegs(E)$. So, by definition of $\Cheegs(E)$, we find out that
\begin{align*}
	\Ps(U)-\Lambda|U|\geq 0
\end{align*}
for all $U\in\measurablesets$ s.t. $U\subseteq E$. This, recalling the equivalence between \ref{Pb:GeometricProblem} and \eqref{Pb:GeomProbConvexDatum}, leads to $\emptyset\in\GeometricProblemSolutions{E, \Lambda}$ and consequently to $0\in\FunctionalProblemSolutions{\characteristic{E}, \Lambda}$. So $\Lambdastardown\geq\Cheegs(E)$.

Now, we fix $\Lambda>\Cheegs(E)$ and we notice that, by definition of $\Cheegs(E)$, there exists $U=U(\Lambda)\in\measurablesets$ contained in $E$ such that
\begin{align*}
	\Ps(U)-\Lambda|U|< 0,
\end{align*} 
which implies $0\notin\FunctionalProblemSolutions{\characteristic{E}, \Lambda}$, thus proving the claim.

We now show that, for $0<\Lambda<\Lambdastardown$ we have $\GeometricProblemSolutions{E,\Lambda}=\{\emptyset\}$, so $\FunctionalProblemSolutions{\characteristic{E}, \Lambda}=\{0\}$. 
Indeed, if $\lambda>0$ is such that $0\in\FunctionalProblemSolutions{E, \lambda}$, which implies $\emptyset\in\GeometricProblemSolutions{E,\lambda}$, then any $0<\Lambda<\lambda$ satisfies,
\begin{align*}
	\Ps(U)-\Lambda|U|\geq \Ps(U)-\lambda|U|\geq 0
\end{align*} 
for every measurable $U\subseteq E$. So, $\emptyset\in\GeometricProblemSolutions{E,\Lambda}$. Now, assume $U\in\GeometricProblemSolutions{E,\Lambda}$ and $U\neq\emptyset$. Then, by the equivalence between \ref{Pb:GeometricProblem} and \eqref{Pb:GeomProbConvexDatum}, it follows that $\Lambda$ is equal to the s-Cheeger constant of $E$, which is a contradiction. This proves $\ref{item:Sone}$.

If $\Lambda=\Lambdastardown$, thanks to the equivalence between \ref{Pb:GeometricProblem} and \eqref{Pb:GeomProbConvexDatum}, 
we deduce $\GeometricProblemSolutions{E, \Lambda}=\sCheegerClass(E)$. From \Cref{prop:equivalenceFuncGeomPBs} and \Cref{cor:BINequivalenceFuncGeomPBs}
it then follows that $\FunctionalProblemSolutions{\characteristic{E},\Lambda} = \Set{u\in\Wspace{s}{1}(\Renne;[0,1])}{\{u>t\} \in \sCheegerClass(E) \text{ for all }t\in [0,1)}$.
This proves $\ref{item:Stwo}$.

$\ref{item:Sthree}$ follows directly from the definition of $\Lambdastardown=\Cheegs(E)$.
\end{proof}

If the set $E$ is convex and $s$-calibrable, we can characterize the solution also for $\Lambda>\Cheegs(E)$.

\begin{thm}\label{thm:convexsCalibrable}
	Let $s\in(0,1)$ and $E\subset\Renne$ be a bounded convex set which is $s$-calibrable. 
	Then, $\FunctionalProblemSolutions{\characteristic{E},\Lambda}=\{\characteristic{E}\}$ for all $\Lambda>\Cheegs(E)$.
\end{thm}

\begin{proof}
	Since $E$ is bounded and convex, the equivalence between \ref{Pb:GeometricProblem} and \eqref{Pb:GeomProbConvexDatum} holds for every $\Lambda>0$.
	Let $\Lambda>\Cheegs(E)$. Being $E$ s-calibrable, from \eqref{eq:sCalibrable} and \Cref{thm:convexsCheeger} we deduce that $\GeometricProblemSolutions{E, \Cheegs(E)}=\{E\}$.
	Now, for $V\in\measurablesets$ s.t. $V\subsetneqq E$ we have,
	\begin{align*}
		&\Ps(E)-\Lambda|E|=
		\Ps(E)-\Cheegs(E)|E|-(\Lambda-\Cheegs(E))|E|\\
		&\leq\Ps(V)-\Cheegs(E)|V|-(\Lambda-\Cheegs(E))|E|\\
		&<\Ps(V)-\Cheegs(E)|V|-(\Lambda-\Cheegs(E))|V|=\Ps(V)-\Lambda|V|.		
	\end{align*}
Therefore, $\GeometricProblemSolutions{E, \Lambda}=\{E\}$ and, consequently, $\FunctionalProblemSolutions{\characteristic{E},\Lambda}=\{\characteristic{E}\}$, by \Cref{cor:BINequivalenceFuncGeomPBs}. 	
\end{proof}

\begin{cor}\label{cor:ConvexSCalibrableRegular}
	Let $s\in(0,1)$. If $E\subset\Renne$ is a bounded convex set which is $s$-calibrable, then $\partial E$ is of class $C^1$. 
\end{cor}
\begin{proof}
	\newcommand{\Eblowup}{\tilde{E}}
	By \Cref{thm:convexsCalibrable} and \Cref{cor:regularity}, we get that $\partial E$ is of class $C^1$ outside a closed singular set $S$ of Hausdorff dimension at most $n-3$. Since $E\in\GeometricProblemSolutions{E, \Lambda}$ for $\Lambda>\Cheegs(E)$, $E$ is a nonlocal almost minimal set for $\Ps$ \cite[see][Theorem 3.3]{CesNov17}. 
	Now we fix $x\in\partial E$ and we consider a blow-up of $E$ at $x$, which we denote by $\Eblowup$.	
	By \cite[Theorem 4.7]{CapGui}, we deduce that $\Eblowup$ is a minimizer for $\Ps$ and a cone. 
	Moreover, being $E$ convex, $\Eblowup$ is convex. Therefore, by applying \cite[Theorem 1.1]{FigVal} to $\Eblowup$, we deduce that the cone $\Eblowup$ is a half-space. Then, $x$ does not belong to $S$; otherwise $\Eblowup$ would be a singular cone \cite[cf.][Definition 9.5]{CafRoqSav}.	  
\end{proof}

\section{Acknowledgements}
The author wishes to thank Matteo Novaga for his comments and his support throughout the preparation of the whole work.
The author is member of the INdAM GNAMPA.
This research did not receive any specific grant from funding agencies in the public, commercial, or not-for-profit sectors.

\bibliographystyle{abbrvnat}	
\bibliography{bibliografia}

	
	
\end{document}